\documentclass{amsart}
\usepackage{enumerate}
\usepackage{amssymb,latexsym}
\usepackage{graphicx}
\usepackage{verbatim}
\usepackage{xcolor} 
\usepackage{hyperref}
\usepackage[colorinlistoftodos]{todonotes}
 \usepackage{xcolor}
\newtheorem{theorem}{Theorem}

\newtheorem{proposition}[theorem]{Proposition}
\newtheorem{lemma}[theorem]{Lemma}
\newtheorem{corollary}[theorem]{Corollary}
\newtheorem{remark}[theorem]{Remark}

\newcommand{\R}{\mathbb{R}}
\newcommand{\Ee}{\mathbb {E}}
\newcommand{\Q}{\mathbb{Q}}

\newcommand{\Sf}{\mathbb{S}}

\newcommand{\Le}{\mathbb{L}}
\newcommand{\Ve}{\mathbb{V}}

\newcommand{\Hy}{\mathbb{H}}
\newcommand{\Oes}{\mathbb{O}}

\newcommand{\spa}{\mbox{span}}

\def\<{\langle}

\def\>{\rangle}

\def\va{\varphi}

\def\bea{\begin{eqnarray*} }
\def\eea{\end{eqnarray*} }
\def\be{\begin{equation} }
\def\ee{\end{equation} }

\def\proof{\noindent{\it Proof: }}
\def\qed{\ifhmode\unskip\nobreak\fi\ifmmode\ifinner
\else\hskip5 pt \fi\fi\hbox{\hskip5 pt \vrule width4 pt
height6 pt  depth1.5 pt
\hskip 1pt }}
\begin{document}

\title{ Umbilical submanifolds of $\Hy^k\times \Sf^{n-k+1}$.}
\maketitle

\begin{center}
\author{M. I. Jimenez        \and
        R. Tojeiro$^*$}  
        \footnote{Corresponding author}
      \footnote{The first author is supported by CAPES-PNPD Grant 88887.469213/2019-00. The second author is partially supported by Fapesp grant 2016/23746-6 and CNPq grant 303002/2017-4.\\
      Data availability statement: Not applicable.}
\end{center}
\date{}

\begin{abstract}
 In this article we complete the classification of the umbilical submanifolds of a Riemannian product of space forms, addressing the case of a  conformally flat product $\Hy^k\times \Sf^{n-k+1}$, which has not been covered in previous works on the subject. We show that  there exists precisely a $p$-parameter family of congruence classes of  umbilical submanifolds of  $\Hy^k\times \Sf^{n-k+1}$ with substantial codimension~$p$, which we prove to be at most $\mbox{min}\,\{k+1, n-k+2\}$. We study more carefully the cases of codimensions one and two and exhibit, respectively, a  one-parameter family and a two-parameter family (together with three extra one-parameter families) that contain precisely one representative of each congruence class of such submanifolds. In particular, this yields another proof of the classification of all (congruence classes of) umbilical submanifolds of $\Sf^n\times \R$, and provides a similar classification for the case of $\Hy^n\times \R$. We determine all possible topological types, actually, diffeomorphism types, of a complete umbilical submanifold of $\Hy^k\times \Sf^{n-k+1}$. We also show that umbilical submanifolds of the product model of $\mathbb{H}^k\times \mathbb{S}^{n-k+1}$ can be regarded as rotational submanifolds in a suitable sense, and explicitly describe their profile curves when $k=n$. As a consequence of our investigations, we prove that every conformal diffeomorphism of $\Hy^k\times \Sf^{n-k+1}$ onto itself is an isometry.
\end{abstract}

\vspace{2ex}


\date{}
\maketitle

\section{Introduction}  
An isometric immersion $f\colon M\to N$  of a Riemannian manifold $M$ into another Riemannian manifold $N$ is \emph{umbilical} if there exists a normal vector field $H$ along $f$ such that its second fundamental form 
$\alpha_f\in \Gamma(T^*M\otimes T^*M\otimes N_f M)$
with values in the normal bundle is given by $\alpha_f(X,Y)=\<X,Y\>H$ 
for all $X,Y\in \mathfrak{X}(M)$. 
Geometrically, this roughly means that $f(M)$ is equally curved in $N$ in all tangent directions. Umbilical submanifolds are the simplest possible submanifolds of a Riemannian manifold after the totally geodesic ones, for which $\alpha_f$ vanishes identically.

It is apparent that a generic  Riemannian manifold does not carry any umbilical submanifold of dimension $n\geq 2$. For instance, it was shown by Souam and Toubiana \cite{st} that of the eight three-dimensional model geometries of Thurston, only the spaces of constant curvature, the product spaces $\Sf^2\times \R$ and $\Hy^2\times \R$, and the $\emph{Sol}$ geometry admit umbilical surfaces, in the latter case a unique non totally geodesic one, up to ambient isometries. Here $\Sf^n$ and $\Hy^n$ stand for the sphere and hyperbolic space of dimension $n$, respectively. On the other hand, by a result of Leung and Nomizu \cite{LN} (see also Section $1.7$ of \cite{dt}), if $N$ is a Riemannian manifold that satisfies the axiom of $r$-spheres, that is, for any point $x\in N$ and for any integer $r\geq 2$, there always exists an extrinsic sphere  of $N$ tangent at $x$ to any prescribed $r$-dimensional subspace $V\subset T_xN$, then $N$ must have constant sectional curvature. Recall that an extrinsic sphere of a Riemannian manifold $N$ is an umbilical submanifold of $N$ whose mean curvature vector field is parallel in the normal connection. It is also pointed out in \cite{LN} that if the requirement  on the mean curvature vector field  is dropped in the axiom of $r$-spheres, then the corresponding weaker axiom for $n > 4$ and $3 < r < n$  implies that M is conformally flat. 

 Besides the spaces of constant curvature, there are few Riemannian manifolds for which a classification of all its umbilical submanifolds is known. For a general symmetric space $N$, it was shown by Nikolayevsky (see Theorem $1$ of \cite{Ni}) that any umbilical submanifold of  $N$ is an umbilical submanifold of a product of space forms totally geodesically embedded in $N$. This was one of our motivations for studying  umbilical submanifolds of a product of space forms.
 
For submanifolds of dimension $m\geq 3$ of a product $\Q_{k_1}^{n_1}\times \Q_{k_2}^{n_2}$ of space forms  whose curvatures satisfy $k_1+k_2\neq 0$, the problem was reduced in \cite{mt1} to the classification of $m$-dimensional umbilical submanifolds  with codimension two of $\Sf^n\times \R$ and $\Hy^n\times \R$. The case of  $\Sf^n\times \R$  was carried out in \cite{mt2}, extending previous results in \cite{st} and \cite{vv} for hypersurfaces. The case $m=2$ was  studied in \cite{ot}, where it was shown that, in addition to the examples that appear already in higher dimensions, there are precisely two distinct two--parameter families of complete embedded flat umbilical surfaces that lie substantially in $\Hy_{k}^{3}\times \R^2$ and $\Hy_{k_1}^{3}\times \Hy_{k_2}^{3}$, respectively.

   In this article we address the classification of the umbilical submanifolds of the products  $\Q_{k_1}^{n_1}\times \Q_{k_2}^{n_2}$ of space forms  in the cases that have not been covered in  \cite{mt1}, \cite{mt2} and  \cite{ot}, namely, the cases in which either $k_1+k_2= 0$ or $k_1<0$ and $n_2=1$. Together with the case in which $k_1>0$ and $n_2=1$, already considered in \cite{mt2}, these are precisely the products of Riemannian manifolds that are conformally flat. In fact, there exists a global conformal diffeomorphism of  
$\Hy^k\times \Sf^{n-k+1}$ onto the complement $\R^{n+1}\setminus \R^{k-1}$ of a $(k-1)$-dimensional subspace of $\R^{n+1}$ (see Section $3$). Since umbilical submanifolds are preserved under a conformal change of the metric of the ambient manifold, one might consider this to imply already a classification of the umbilical submanifolds of $\Hy^k\times \Sf^{n-k+1}$: any umbilical isometric immersion $f\colon M^m \to \Hy^k\times \Sf^{n-k+1}$ is produced by composing  the standard  inclusion into $\R^{n+1}\setminus \R^{k-1}$ of an open subset of either $\Sf^m$ or $\R^m$ with the inverse of such diffeomorphism. 
  
  However, for a complete classification of the umbilical submanifolds of  $\Hy^k\times \Sf^{n-k+1}$, one should be able to decide when two such umbilical isometric immersions $f, \tilde f\colon M^m \to \Hy^k\times \Sf^{n-k+1}$ are congruent, that is, when there exists an isometry $\Phi\colon \Hy^k\times \Sf^{n-k+1}\to \Hy^k\times \Sf^{n-k+1}$ such that $\tilde f=\Phi\circ f$. In other words, a complete classification of the umbilical submanifolds of  $\Hy^k\times \Sf^{n-k+1}$ requires a description of all congruence classes of such  submanifolds as well as the knowledge of a representative of each of them.  Other natural questions one may pose are: what is the maximal substantial codimension of an umbilical submanifold of $\Hy^k\times \Sf^{n-k+1}$, that is, the maximal codimension of  an umbilical submanifold  of $\Hy^k\times \Sf^{n-k+1}$ that is not contained in a totally geodesic submanifold? What are its geometry and topology?

     We now outline the contents of the paper and the results we have obtained while investigating these problems. In Section $2$, we recall some basic facts on the model of Euclidean space $\R^{n+1}$ as a hypersurface of the light cone in the Lorentz space 
$\Le^{n+3}$, a convenient setting to deal with such problems. In Section $3$, we introduce the conformal model $(\R^{n+1}\setminus \R^{k-1}, g)$ of $\Hy^k\times \Sf^{n-k+1}$, and describe its isometries (Theorem~\ref{isometries}) and totally geodesic submanifolds (Proposition \ref{prop:totgeo}). As a consequence of independent interest of Theorem \ref{isometries}, we see that every conformal diffeomorphism of $\Hy^k\times \Sf^{n-k+1}$ is an isometry (Corollary \ref{cor:iso}), which extends a similar result by Souam and Toubiana \cite{st} for the particular cases of the product spaces $\mathbb{S}^2\times \mathbb{R}$ and $\mathbb{H}^2\times \mathbb{R}$ as well as a well-known property of hyperbolic space. The conformal model $(\R^{n+1}\setminus \R^{k-1}, g)$ of $\Hy^k\times \Sf^{n-k+1}$ also allows us to easily determine all possible topological types, actually, diffeomorphism types, of a complete umbilical submanifold of $\Hy^k\times \Sf^{n-k+1}$ (Proposition \ref{prop:topology}).

In Section $4$ we address the congruence problem for umbilical submanifolds  of 
 $\Hy^k\times \Sf^{n-k+1}$. First we show that any such  submanifold  with codimension $p$ arises as the intersection of $\Hy^k\times \Sf^{n-k+1}\subset \Le^{n+3}$ with a codimension $p$ time-like subspace of $\Le^{n+3}$ (Proposition \ref{intersection}). Then we obtain an algebraic criterion to decide when two such submanifolds are congruent and use it to prove that there exists precisely a $p$-parameter family of congruence classes of  umbilical submanifolds of $\Hy^k\times \Sf^{n-k+1}$ with substantial codimension~$p$ (Theorem~\ref{thm:congclasses}). Along the way, we show that the substantial codimension of an  umbilical submanifold  of $\Hy^k\times \Sf^{n-k+1}$ is at most $\mbox{min}\,\{k+1, n-k+2\}$ (Corollary~\ref{substantial}). Then we study more carefully the cases of codimensions one and two. We give necessary and sufficient conditions for two umbilical submanifolds  of $(\R^{n+1}\setminus \R^{k-1}, g)$ with codimensions one and two to be congruent in terms of their Euclidean data (Propositions \ref{eq:congcond} and \ref{eq:congcondb}, respectively) and exhibit, respectively, a  one-parameter family and a two-parameter family (together with three extra one-parameter families) that contain precisely one representative of each congruence class of such submanifolds (Corollaries \ref{cor:hypcase} and \ref{cor:cod2case}, respectively). In particular, we recover the main result of \cite{mt1} for the particular case of $\mathbb{S}^n\times \mathbb{R}$ (Corollary \ref{cor:cod2snr}) and obtain, as a particular case of Corollary \ref{cor:cod2case}, a similar classification of all (congruence classes of) umbilical submanifolds of $\Hy^n\times \R$. 

       In Section $5$,  we show that umbilical submanifolds of the product model of
$\mathbb{H}^k\times \mathbb{S}^{n-k+1}$ can be regarded as rotational submanifolds in a suitable sense. In particular, umbilical submanifolds of $\mathbb{H}^n\times \mathbb{S}^{1}$ (and hence of its universal cover $\mathbb{H}^n\times \mathbb{R}$) are rotational submanifolds with curves as profiles, and the latter are explicitly described by means of the conformal model of $\mathbb{H}^n\times \mathbb{S}^{1}$, providing an alternative description to that in \cite{mt1} of the profile curves of umbilical submanifolds of $\mathbb{S}^n\times \mathbb{R}$.

  \section{Preliminaries}
 
  In this section we summarize some basic facts on the model of Euclidean space $\R^{n+1}$ as a hypersurface of the light cone in the Lorentz space $\Le^{n+3}$ that make it a useful tool to study  Moebius geometric problems. We refer to Chapter $9$ of \cite{dt} for details.
 
 Let $\Le^{n+3}$ be the $(n+3)$-dimensional Lorentzian space
with the metric induced by the inner product
$$
\<v,w\>=-v_0w_0+v_1w_1+\cdots+v_{n+2}w_{n+2}
$$
for $v=(v_0,\dots,v_{n+2})$ and $w=(w_0,\ldots,w_{n+2})$.
The set of light-like vectors
$$
\Ve^{n+2}=\{v\in\Le^{n+3}:\<v,v\> =0,\;v\neq 0\} 
$$  
is called  the \emph{light cone}\index{Light-cone} of $\Le^{n+3}$. 
For any given $w\in\Ve^{n+2}$, the intersection
$$
\Ee^{n+1}=\{v\in\Ve^{n+2}:\<v,w\>=1\}
$$
of $\Ve^{n+2}$ with the affine hyperplane 
$
\{u\in\Le^{n+3}:\<u,w\>=1\}
$
is a model of the $(n+1)$-dimensional Euclidean space.
Namely, for a fixed $v\in \Ee^{n+1}$ and a linear isometry 
$
C\colon\R^{n+1}\to(\spa\{v,w\})^\perp\subset\Le^{n+3},
$
the map $\Psi\colon\R^{n+1}\to\Le^{n+3}$,
given by
\be\label{eq:Psi}
\Psi(x)=v+Cx-\frac{1}{2}\|x\|^2w,
\ee
is an isometric embedding such that $\Psi(\R^{n+1})=\Ee^{n+1}$. From now on we assume that a triple $(v,w,C)$ has been fixed with $w_0<0$, so that 
$\Ee^{n+1}\subset\Ve^{n+2}_+$, the upper half of $\Ve^{n+2}$.

 For any  hypersphere $\mathcal{S}\subset\R^{n+1}$ 
with (constant) mean curvature $h$ with respect to a  unit normal vector
field $N$, the map
$$
x\in \mathcal{S}\mapsto \Psi_*(x)N(x)+h\Psi(x)\in \mathbb{L}^{n+3}
$$ 
has a constant value $z$, a unit space-like vector satisfying 
$\<\Psi(x),z\>=0$ for all $x\in \mathcal{S}$. It follows that
$
\Psi(\mathcal{S})=\Ee^{n+1}\cap \{z\}^\perp.
$
Observe that $\mathcal{S}$ is an affine hyperplane  if and only if 
$
0=h=\< z,w\>.
$

If $\mathcal{S}$ is a hypersphere  with (Euclidean) center $x_0$
and radius $r$, oriented by its inward pointing unit normal vector field 
$
x\in \mathcal{S}\mapsto N(x)=\frac{1}{r}(x_0-x)
$
with corresponding mean curvature  $h=1/r$, then its associated unit space-like vector is   
$
z=\frac{1}{r} \Psi(x_0)+\frac{r}{2}w.
$
On the other hand, the unit space-like vector associated to an affine  hyperplane $\mathcal{S}$
oriented by a unit normal vector $N$ is 
$
z=CN-cw,
$
where $c\in\R$ is the oriented distance from $\mathcal{S}$ to the origin in $\R^{n+1}$. 

If $\mathcal{S}_1$ and $\mathcal{S}_2$ are hyperspheres or affine hyperplanes with corresponding unit space-like vectors $z_1$ and $z_2$, respectively. and if $N^1_x$ and $N^2_x$ are the
unit normal vectors of $\mathcal{S}_1$ and $\mathcal{S}_2$, respectively, at $x\in
\mathcal{S}_1\cap \mathcal{S}_2$, then 
$
\<N^1_x,N^2_x\>=\< z_1,z_2\>. 
$  
In particular, $\mathcal{S}_1$ and $\mathcal{S}_2$ intersect orthogonally if and only if  
$\<z_1,z_2\>=0$.

More generally, any sphere or affine subspace $\mathcal{S}$ with codimension $k$ of $\mathbb{R}^{n+1}$ is given by 
$\Psi(\mathcal{S})=\Ee^{n+1}\cap V^\perp$ for some space-like subspace 
$V\subset \Le^{n+3}$ of dimension $k$, affine subspaces being characterized by the fact 
that $w\in V^\perp$.  We refer to $V$ in the sequel as the space-like subspace of $\Le^{n+3}$ correspondent to $\mathcal{S}$.

 If $f\colon M^m\to\Ve^{n+2}$ is an immersion of a
differentiable manifold and  $\mu\in C^\infty(M)$ is any positive smooth function, then the
map $h\colon M^m\to\Ve^{n+2}$ given by $ h=\mu f$ is  also an
immersion, and the induced metrics $\<\,,\,\>_f$ and $\<\,,\,\>_h$
are related by 
$
\<\,,\,\>_h=\mu^2\<\,,\,\>_f.
$
As a consequence, any conformal immersion $f\colon M^m\to\R^{n+1}$ 
with conformal factor $\va\in C^\infty(M)$  gives rise to an 
isometric immersion ${\mathcal I}(f)\colon M^m\to\Ve_+^{n+2}$ given by
$
{\mathcal I}(f)=\frac{1}{\va}\Psi\circ f.
$
Conversely, any isometric immersion 
$F\colon M^m\to\Ve_+^{n+2}\smallsetminus\R w$, where  $\R w=\{tw:t<0\}$,
gives rise to a conformal immersion ${\mathcal C}(F)\colon M^m\to\R^{n+1}$  
given by
$
\Psi\circ {\mathcal C}(F) = \Pi\circ F,
$
whose conformal factor is $1/\<F,w\>$.  Here $\Pi\colon\Ve_+^{n+2}\smallsetminus\R w \to\Ee^{n+1}$ denotes the projection 
onto $\Ee^{n+1}$ given by
$ \Pi(u)=u/\<u,w\>$.

In particular, if  $T\in\Oes_1(n+3)$, then
${\mathcal C}(T\circ\Psi)$ is a conformal transformation of $\mathbb{R}^{n+1}$,
and every conformal transformation of $\mathbb{R}^{n+1}$ is given in this way.
For instance, if $R\in\Oes_1(n+3)$ is the reflection 
$
R(u)=u-2\<u,z\>z
$ 
with respect to the hyperplane in $\Le^{n+3}$ orthogonal to a unit 
space-like vector $z$, with $\<z,w\>\neq 0$, then
$
{\mathcal C}(R\circ\Psi)=I
$
is the inversion with respect to the hypersphere such that $\Psi(\mathcal{S})=\Ee^{n+1}\cap\{z\}^\perp$. If  $G\in\Oes_1(n+3)$ satisfies $G(w)=\lambda w$ for some
$\lambda\in(0,+\infty)$, then
$ {\mathcal C}(G\circ\Psi)=L$ 
is a similarity of ratio $\lambda$.  In particular, isometries 
of $\R^{n+1}$ correspond  to the elements of  $\Oes_1(n+3)$ that fix $w$.

\section{The conformal model of $\Hy^k\times \Sf^{n-k+1}$} Let $\R^{k-1}$ be a subspace of dimension $k-1$ of $\R^{n+1}$, which we assume to be 
$$\R^{k-1}=\{(x_1, \ldots, x_{n+1})\in \R^{n+1}\,:\,x_{k}=\cdots=x_{n+1}=0\}.$$
Given $x\in \R^{n+1}$, we write $x=x^T+x^\perp$ according to the orthogonal decomposition $\R^{n+1}= \R^{k-1}\oplus \R^{n-k+2}$. Now let $v, w\in \Le^{n+3}$ be vectors in $\Ve^{n+2}\subset \Le^{n+3}$ with $\<v,w\>=1$, let
$C\colon\R^{n+1}\to(\spa\{v,w\})^\perp\subset\Le^{n+3}$ be a linear isometry, and set
$V_1=C(\R^{k-1})$ and $V_2=C(\R^{n-k+2})$. Let $\Le^{n+3}=W_1\oplus W_2$ be the orthogonal decomposition with $W_1=V_1\oplus \spa\{v,w\}$ and $W_2=V_2$. If $\Psi\colon \R^{n+1}\to \Ee^{n+1}\subset \Le^{n+3}$ is the isometric embedding given by (\ref{eq:Psi}), then the map $\Theta\colon \R^{n+1}\setminus \R^{k-1}\to \Hy^k\times \Sf^{n-k+1}\subset W_1\oplus W_2=\Le^{k+1}\times \R^{n-k+2}=\Le^{n+3}$ defined by
 \begin{eqnarray}\label{eq:Theta}\Theta(x)&=&\|x^{\perp}\|^{-1}\Psi(x)\nonumber\\
 &=&\|x^{\perp}\|^{-1}\left( v-\frac{1}{2}\|x\|^2w+Cx^T+Cx^\perp\right)
 \end{eqnarray} 
is a conformal diffeomorphism with induced metric $g=\varphi^{2}g_0$ whose conformal factor with respect to the Euclidean metric $g_0$ is given by  $\va(x)=\|x^\perp\|^{-1}=(\sum_{i=k}^{n+1} x_i^2)^{-1/2}$, that is, $\varphi(x)$ is  the inverse of the distance from $x$ to $\R^{k-1}$. We refer to $(\R^{n+1}\setminus \R^{k-1}, g)$ as the \emph{conformal model} of $\Hy^k\times \Sf^{n-k+1}$. The above notations will be used throughout the paper.

\subsection{The isometries of $(\R^{n+1}\setminus \R^{k-1}, g)$}

The next result describes the isometries of $(\R^{n+1}\setminus \R^{k-1}, g)$. 

\begin{theorem} \label{isometries} The following assertions on the map $F\colon \R^{n+1}\setminus \R^{k-1}\to \R^{n+1}\setminus \R^{k-1}$ are equivalent for any $1\leq k\leq n+1$:
\begin{itemize}
\item[(i)] $F$ is an isometry with respect to $g$;
\item[(ii)] $\Theta\circ F= T\circ \Theta$, where $T\in O_1(n+3)$ leaves $W_1$ and $W_2$ invariant;
\item[(iii)]  $F=\mathcal{C}(T\circ \Psi)$, where $T\in O_1(n+3)$ 
leaves $W_1$ and $W_2$ invariant;
\item[(iv)] $F$  is a conformal diffeomorphism with respect to the Euclidean metric $g_0$;
\item[(v)] $F$  is a conformal diffeomorphism with respect to $g$; 
\item[(vi)]  $F$ is given by $F=I\circ L$, where $L$ is a composition $L=G\circ H\circ A$ of an orthogonal transformation $A\in O(n+1)$ that leaves $\R^{k-1}$ invariant, a homothety with respect to the origin and a translation by a vector  $v\in \R^{k-1}$, and $I$ is an inversion with respect to a hypersphere centered at a  point of $\R^{k-1}$.
\end{itemize}
\end{theorem}
\proof $(i) \Longrightarrow (ii)$: Let $F$ be an isometry of $(\R^{n+1}\setminus \R^{k-1},g)$. Since $\Theta\colon (\R^{n+1}\setminus \R^{k-1}, g)\to \Hy^k\times \Sf^{n-k+1}$ is an isometry,  then $\Theta\circ F= T\circ \Theta$ for some element $T$ of the isometry group 
$\mbox{Iso}\,(\Hy^k\times \Sf^{n-k+1})= \mbox{Iso}\,(\Hy^k)\times \mbox{Iso}\,(\Sf^{n-k+1})$ of $\Hy^k\times \Sf^{n-k+1}\subset W_1\oplus W_2=\Le^{n+3}$, which is an orthogonal transformation $T\in O_1(n+3)$ that  leaves $W_1$ and $W_2$ invariant.\vspace{1ex}\\
$(ii)\Longrightarrow (iii)$:   Since $\Theta(x)=\|x^{\perp}\|^{-1}\Psi(x)$, then $\Pi\circ \Theta=\Psi$, and hence
$$\Pi\circ T\circ \Psi= \Pi\circ T\circ \Theta=\Pi\circ \Theta\circ F=\Psi\circ F.$$

 That $(iii)$ implies $(iv)$ is clear, as is the equivalence between $(iv)$ and $(v)$. \vspace{1ex}\\
 $(iv)\Longrightarrow (iii)$: The assumption that $F(\R^{n+1}\setminus \R^{k-1})\subset \R^{n+1}\setminus \R^{k-1}$ means that $(F(x))^\perp\neq 0$ whenever $x^\perp\neq 0$. Since
 $\Psi(x)=v-\frac{1}{2}\|x\|^2w+Cx^T+Cx^\perp$
 and  $\Psi\circ F=\Pi\circ T\circ \Psi$, this implies that $\pi_{W_2}(Ty)\neq 0$ whenever $\pi_{W_2}(y)\neq 0$, where $\pi_{W_2}\colon \Le^{n+3}\to W_2=V_2$ is the orthogonal projection.  In other words, $\ker (\pi_{W_2}\circ T)\subset \ker \pi_{W_2}=W_1$, or equivalently, $y\in W_1$ whenever $Ty\in W_1$. Thus $W_1$ is invariant under $T^{-1}$, and hence both $W_2$ and $W_1$ are invariant under $T$. \vspace{1ex}\\
 $(iii)\Longrightarrow (vi)$:   Define $(\bar{v},\bar{w},\bar{C})$ by 
$\bar{v}=T(v),\;\;\bar{w}=T(w)\;\;\mbox{and}\;\;\bar{C}=T\circ C.
$
If $\bar{w}=\lambda w$ for some $\lambda\in(0,+\infty)$, then ${\mathcal C}(T\circ\Psi)=L$
is a similarity of ratio $\lambda$. Moreover, since $T$ leaves $W_1$ and $W_2$ invariant, from $\Psi\circ L=\Pi\circ T\circ \Psi$ it follows that $L$ leaves $\R^{k-1}$ and 
$\R^{n-k+2}$ invariant. Writing $L=G\circ H\circ A$, where  $A\in O(n+1)$ is  an orthogonal transformation, $H$ is a homothety and $G$ is a translation by a vector  $y\in \R^{n+1}$, it follows that $y\in \R^{k-1}$ and that $A$  leaves  $\R^{k-1}$  invariant. 

Otherwise, consider the reflection 
$
R(u)=u-2\<u,z\>z
$
determined by the unit space-like vector
${\displaystyle 
z=\frac{1}{\<\bar{w},w\>}\bar{w}+\frac{1}{2}w}$, 
and let $G\in\Oes_1(n+3)$ be given  by
$$
G(w)=R(\bar{w})=-\frac{1}{2}\<\bar{w},w\>w,\;\;G(v)=R(\bar{v})
\;\;\mbox{and}\;\;G\circ C=R\circ\bar{C}.
$$
 Then $R\circ G$ takes $w$ to $\bar{w}$,
$v$ to $\bar{v}$ and $R\circ G\circ C= \bar{C}$, hence $R\circ G=T$.
The map 
${\mathcal C}(R\circ\Psi)=I$ is an inversion with respect to the hypersphere 
of unit radius such that $\Psi(\mathcal{S})=\Ee^{n+1}\cap\{z\}^\perp$, whereas ${\mathcal C}(G\circ\Psi)=L$  
is a  similarity  of ratio $\lambda=-(1/2)\<\bar{w},w\>$. Moreover, since $T$ leaves $W_1$ invariant, then $\bar{w}$, and hence $z$, belongs to $W_1$. Therefore, the hypersphere $\mathcal{S}$ such that $\Psi(\mathcal{S})=\Ee^{n+1}\cap\{z\}^\perp$ is centered at a  point of $\R^{k-1}$. Finally, from
$$\Pi\circ T\circ\Psi=\Pi\circ R\circ G\circ\Psi=\Pi\circ R\circ\Psi\circ L=\Psi \circ I\circ L,$$
it follows that $F=I\circ L$.\vspace{1ex}\\
$(vi)\Longrightarrow (i)$: Let $x_0\in \R^{k-1}$. The inversion with respect to $\Sf^{n}(x_0;r)$ is given by
$$ I(x)=x_0+r^2\frac{x-x_0}{\|x-x_0\|^2},$$
which is a conformal map with conformal factor ${\displaystyle \rho(x)= \frac{r^2}{\|x-x_0\|^2}}$. In particular,
$$ (I(x))^\perp=r^2\frac{x^\perp}{\|x-x_0\|^2}.$$
Therefore
\begin{eqnarray*} g(I(x))(I_*(x)v, I_*(x)w)
&=&\frac{1}{\|(I(x))^\perp\|^2}g_0(I(x))(I_*(x)v, I_*(x)w)\\
&=&\frac{\|x-x_0\|^4}{r^4\|x^\perp\|^2}g_0(I(x))(I_*(x)v, I_*(x)w)\\
&=&\frac{\|x-x_0\|^4}{r^4\|x^\perp\|^2}\frac{r^4}{\|x-x_0\|^4}g_0(x)(v, w)\\
&=&g(x)(v, w).
 \end{eqnarray*}
 The verification that an orthogonal transformation $A\in O(n+1)$ that leaves $\R^{k-1}$ invariant, a homothety  and a translation by a vector $y\in \R^{k-1}$ are also isometries with respect to $g$ is straightforward. \vspace{1ex}\qed
 
  For instance,  an  inversion with respect to a hypersphere of $\R^{n+1}$ centered on $\R^{k-1}$  acts on each $k$-dimensional half-space $\R^k_+$ having $\R^{k-1}$ as boundary as an hyperbolic isometry of $\R^k_+$ regarded as the half-space model of $\Hy^k$ with $\R^{k-1}$ as asymptotic boundary. Each such $\Hy^k$ is a totally geodesic submanifold of $(\R^{n+1}\setminus \R^{k-1}, g)$, which corresponds under  
$\Theta$ to a slice $\Hy^k\times \{x\}$ of $\Hy^k\times \Sf^{n-k+1}$ (see the next subsection for a description of all totally geodesic submanifolds of  $(\R^{n+1}\setminus \R^{k-1}, g)$). \vspace{1ex}

Theorem \ref{isometries} has the following consequence of independent interest.

\begin{corollary}\label{cor:iso} Any conformal diffeomorphism of $\Hy^k\times \Sf^{n-k+1}$ onto itself, $1\leq k\leq n+1$, is an isometry.
\end{corollary}

\begin{remark}\emph{The assertion of Corollary \ref{cor:iso} in the particular case of  the product spaces $\mathbb{S}^2\times \mathbb{R}$ and $\mathbb{H}^2\times \mathbb{R}$ was proved in \cite{st}, while for $k=n+1$ it reduces to a well-known property of the hyperbolic space $\Hy^n$.}
\end{remark}

\subsection{The totally geodesic submanifolds of  $(\R^{n+1}\setminus \R^{k-1}, g)$ }

In this subsection we describe the totally geodesic submanifolds of  $(\R^{n+1}\setminus\R^{k-1},g)$ and their correspondent space-like subspaces of $\Le^{n+3}$.

\begin{proposition} \label{prop:totgeo}
The totally geodesic submanifolds of $(\R^{n+1}\setminus\R^{k-1},g)$ are
\begin{itemize}
    \item [(i)] Affine subspaces $x_0+\R^{l-1}\oplus\R^{m-l+1}$ of $\R^{n+1}=\R^{k-1}\oplus\R^{n-k+2}$, with $l\leq k$ and $x_0\in \R^{k-1}$, where $\R^{l-1}\subset \R^{k-1}$ and $\R^{m-l+1}\subset\R^{n-k+2}$ are linear subspaces.
    \item [(ii)] Hyperspheres of affine subspaces as above centered at points of $\R^{l-1}$. 
\end{itemize}
\end{proposition}
\proof It is well-known that if $f\colon M^n\to\tilde M^m$ is an immersion and 
 ${g}_0, {g}_1$ are conformal metrics on $\tilde M^m$, then the second fundamental forms of  
$f_j=f\colon (M^n,f^*g_j)\to (\tilde M^m, {g}_j)$, $0\leq j\leq 1$, are related by
\begin{equation}\label{eq:rel}
\alpha^{f_1}(x)(X,Y)=\alpha^{f_0}(x)(X,Y)-\frac{1}{\varphi(f(x))}f^*g_0(X,Y)(\mbox{grad}_0\varphi(f(x)))_{N_fM(x)}
\end{equation}
for all $x\in M^n$ and $X,Y\in T_xM$, where $\varphi\in C^{\infty}({M})$ is 
the  conformal factor of $g_1$ with respect to $g_0$ and $\mbox{grad}_0 \varphi$ 
denotes the gradient of $\varphi$ with respect to $g_0$. Therefore, $f_1$ is totally geodesic if and only if either $f_0$ is totally geodesic and   
$(\mbox{grad}_0\varphi(f(x)))_{N_fM(x)}=0$ for all $x\in M$, 
or $f_0$ is umbilical with mean curvature vector
$$H(x)= \frac{\mbox{grad}_0\varphi(f(x))_{N_fM(x)}}{\varphi(f(x))}$$ for all $x\in M$.

 For $\tilde M^m=\R^{n+1}\setminus \R^{k-1}$,  $g_0$ the standard Euclidean metric and $g_1:=g=\varphi^2g_0$, with $\varphi(x)=\|x^\perp\|^{-1}=g_0(x^\perp, x^\perp)^{-1/2}$, 
we have
 $\mbox{grad}_0\varphi(f(x))=\frac{f(x)^\perp}{\|f(x)^\perp\|^2}$ for all $x\in \R^{n+1}\setminus \R^{k-1}$. Hence, if $f\colon M^n\to \R^{n+1}\setminus \R^{k-1}$ is an immersion, then 
 $f_1=f\colon (M^m,f^*g)\to (\R^{n+1}\setminus \R^{k-1}, {g})$ is totally geodesic 
 if and only if either  $f_0=f\colon (M^m,f^*g_0)\to (\R^{n+1}\setminus \R^{k-1}, {g}_0)$ is totally geodesic and  $(f(x)^\perp)_{N_fM(x)}=0$ for all $x\in M^m$, or $f_0$ is umbilical and its  mean curvature vector is given by
$H(x)= -\frac{(f(x)^\perp)_{N_fM}}{\|f(x)^\perp\|^2}$ for all $x\in M^m$.

In the former case, $f(M)$ is an open subset of an affine subspace $x_0+\R^{l-1}\oplus\R^{m-l+1}$ of $\R^{n+1}=\R^{k-1}\oplus\R^{n-k+2}$, with $l\leq k$  and $x_0\in \R^{k-1}$, where $\R^{l-1}\subset \R^{k-1}$ and $\R^{m-l+1}\subset\R^{n-k+2}$ are linear subspaces.

In the latter case, suppose first that $m=n$, in which case $f(M)$ is an open subset of a hypersphere centered at $x_0\in \R^{n+1}$, whose mean curvature vector is 
$$H(x)=\frac{1}{r}N(x),$$
where $r=\|f(x)-x_0\|$ is its radius and $N(x)=\frac{1}{r}(x_0-f(x))$ is its inward pointing unit normal vector field.
Thus, $H(x)= -\frac{(f(x)^\perp)_{N_fM}}{\|f(x)^\perp\|^2}$ if and only if
$$
\frac{1}{r}=-\frac{\<f(x)^\perp,N\>}{\|f(x)^\perp\|^2}=\frac{\<f(x)^\perp,f(x)-x_0\>}{r\|f(x)^\perp\|^2}.
$$
which is equivalent to
$$
 \<f(x)^\perp,x_0\>=0
$$
for all $x\in M^n$, that is $x_0^\perp=0$.
It follows that $f(M)$ is an open subset of a hypersphere centered at $x_0\in \R^{k-1}$. For arbitrary $m$, $f(M)$ is the intersection of such a hypersphere with an affine subspace $x_0+\R^{l-1}\oplus\R^{m-l+1}$ of $\R^{n+1}=\R^{k-1}\oplus\R^{n-k+2}$, with $l\leq k$ and $x_0\in \R^{k-1}$, where $\R^{l-1}\subset \R^{k-1}$ and $\R^{m-l+1}\subset\R^{n-k+2}$ are linear subspaces, and hence it is an open subset of a hypersphere of such an affine subspace centered at a point of $\R^{l-1}$. \vspace{1ex}\qed

For instance, hyperspheres of $\R^{n+1}$ centered at points of $\R^{k-1}$ and affine hyperplanes $x_0+ \R^{k-2}\oplus \R^{n-k+2}$ correspond to totally geodesic hypersurfaces $\Hy^{k-1}\times\Sf^{n-k+1}$.

\begin{corollary} \label{cor:totgeosub} The totally geodesic submanifolds of  $(\R^{n+1}\setminus\R^{k-1},g)$ are the (open subsets of) spheres and affine subspaces of $\R^{n+1}$ correspondent to the space-like subspaces $\tilde{V}\subset \Le^{n+3}$ of the form  
$\tilde{V}=\tilde{V}_1\oplus\tilde{V}_2$ with $\tilde{V}_1\subset W_1$ and $\tilde{V}_2\subset W_2$.
\end{corollary}
\proof Hyperplanes  $\R^{k-1}\oplus\R^{n-k+1}$ and affine hyperplanes  $x_0+\R^{k-2}\oplus\R^{n-k+2}$ of $\R^{n+1}=\R^{k-1}\oplus\R^{n-k+2}$ correspond, respectively,  to unit space-like vectors of the type $\eta=C(N)\in C(\R^{n-k+2})= W_2$ or $\zeta=-cw+C(N)\in \spa\{v,w\}\oplus C(\R^{k-1})=W_1$,  whereas hyperspheres centered at points $x_0\in \R^{k-1}$ correspond to unit space-like vectors of the type
 $$\frac{1}{r}\Psi(x_0)+\frac{r}{2}w=
 \frac{1}{r}v+\frac{1}{2}\left(r-\frac{\|x_0\|^2}{r}\right)w+Cx_0\in  \spa\{v,w\}\oplus C(\R^{k-1})=W_1.$$
  Therefore, totally geodesic submanifolds of  $(\R^{n+1}\setminus\R^{k-1},g)$, as orthogonal intersections of hyperplanes, affine hyperplanes and hyperspheres as above, correspond to space-like subspaces that are spanned by orthogonal space-like vectors that belong to either $W_1$ or $W_2$, and hence are space-like subspaces of the type $\tilde{V}=\tilde{V}_1\oplus\tilde{V}_2$ with $\tilde{V}_1\subset W_1$ and $\tilde{V}_2\subset W_2$. \qed
  
\subsection{The topology of an umbilical submanifold of  $\Hy^k\times \Sf^{n-k+1}$}
The topology, actually the diffeomorphism type, of a complete umbilical submanifold of  $\Hy^k\times \Sf^{n-k+1}$, is easily determined by means of the conformal model  $(\R^{n+1}\setminus\R^{k-1},g)$ of $\Hy^k\times \Sf^{n-k+1}$.

\begin{proposition}\label{prop:topology}  Any umbilical isometric immersion $f\colon M^m\to \Hy^k\times \Sf^{n-k+1}$, $1\leq k\leq n$, of a complete Riemannian manifold is an embedding, and $M^m$ is diffeomorphic 
 to either $\Sf^m$ or $\R^m$ if $k=1$, or else to  $\Sf^{m-d}\times \R^{d}$, $k-1-n+m\leq d\leq \min\{k-1,m\}$, if $k\geq 2$.
\end{proposition}
\proof We may assume that $f$ is an isometric immersion into $(\R^{n+1}\setminus\R^{k-1},g)$. Since umbilicity of an isometric immersion is invariant under conformal changes of the ambient space by (\ref{eq:rel}), then there exists a standard isometric embedding $\bar f\colon \bar{M}^m \to (\R^{n+1}, g_0)$, where $\bar{M}^m$ stands for $\Sf^m$ or $\R^{m}$, and a diffeomorphism $\nu\colon M^m\to \bar{M}^m\setminus \bar f^{-1}(\R^{k-1})$
such that $f=\bar f\circ \nu$. If $\bar{M}^m\setminus \bar f^{-1}(\R^{k-1})$ is not connected, we consider $\nu\colon M^m\to U\subset\bar{M}\setminus \bar f^{-1}(\R^{k-1})$, where $U$ is a connected component. Therefore, $f$ is an embedding and $M^m$ is diffeomorphic to $\bar{M}^m\setminus\bar f^{-1}(\R^{k-1})$, which, for $\bar{M}^m=\Sf^m$, is either $\Sf^m$ itself, if $\bar f(\Sf^m)\cap \R^{k-1}=\emptyset$,  $\Sf^m\setminus \{p\}$, if $\bar f(\Sf^m)$ is tangent to $\R^{k-1}$ at $p\in \R^{k-1}$, or $\Sf^m\setminus \Sf^{d-1}$, if $\bar f(\Sf^n)$ intersects $\R^{k-1}$
along a $(d-1)$-dimensional sphere $\Sf^{d-1}$, $k-1-n+m\leq d\leq \min\{k-1,m\}$ (here $\Sf^0$ consists of two points). Clearly, in the second possibility, $M^m$ is diffeomorphic to $\R^m$. In the last one, notice that $\Sf^m\setminus \Sf^{d-1}$ is 
diffeomorphic to $\R^{m}\setminus \R^{d-1}$ by a stereographic projection with respect to a point of $\Sf^{d-1}$, and the latter, in turn, is diffeomorphic to $\Hy^{d}\times \Sf^{m-d}$ by the diffeomorphism $\Theta$ given by (\ref{eq:Theta}), or equivalently, to $\Sf^{m-d}\times \R^{d}$, as stated. The proof when $\bar{M}^m=\R^m$ is similar.\qed

\begin{remark}\emph{It is also clear from the proof of Proposition \ref{prop:topology} that
if $f\colon M^m\to \Hy^k\times \Sf^{n-k+1}$ is an umbilical isometric immersion, then $f(M)$ is an open subset of (the image of)  an umbilical isometric embedding $f\colon \bar{M}^m\to \Hy^k\times \Sf^{n-k+1}$ of a  complete Riemannian manifold $\bar{M}^m$.}
\end{remark}

\section{The congruence problem}

In this section we address the congruence problem for umbilical submanifolds  of $\Hy^k\times \Sf^{n-k+1}$. We obtain in the next subsection an algebraic criterion to decide when two such submanifolds are congruent and then apply it in the subsequent subsections to umbilic submanifolds with codimension one or two of the conformal model $(\R^{n+1}\setminus\R^{k-1},g)$ of $\Hy^k\times \Sf^{n-k+1}$.

\subsection{The algebraic criterion}

 First we show how any umbilical submanifold  of the product model $\Hy^k\times \Sf^{n-k+1}$ arises.
 
 \begin{proposition}\label{intersection} The umbilical submanifolds  of $\Hy^k\times \Sf^{n-k+1}$ with codimension $p$ are the intersections of $\Hy^k\times \Sf^{n-k+1}\subset \Le^{n+3}$ with the codimension $p$ time-like subspaces of $\Le^{n+3}$.
 \end{proposition}
 \proof Since umbilical submanifolds are preserved under a conformal change of the metric of the ambient manifold, the umbilical submanifolds  of $(\R^{n+1}\setminus\R^{k-1},g)$ are the open subsets of the intersections $\mathcal{S}\cap (\R^{n+1}\setminus\R^{k-1})$ with $(\R^{n+1}\setminus\R^{k-1})$ of the spheres or affine subspaces $\mathcal{S}\subset \mathbb{R}^{n+1}$. Therefore, the umbilical submanifolds  of $\Hy^k\times \Sf^{n-k+1}\subset \Le^{n+3}$ with codimension $p$ are the open subsets of the images $\Theta(\mathcal{S}\cap (\R^{n+1}\setminus\R^{k-1}))$ under $\Theta$ of such intersections. As discussed in Section $2$, for each sphere or affine subspace $\mathcal{S}\subset \mathbb{R}^{n+1}$ there exists a space-like subspace $V\subset  \Le^{n+3}$ with dimension $p$ such that $\Psi(\mathcal{S})=\mathbb{E}^{n+1}\cap V^\perp$. Since $\Theta(x)=\|x^\perp\|^{-1}\Psi(x)$ for all $x\in \R^{n+1}\setminus\R^{k-1}$ and $\Theta((\R^{n+1}\setminus\R^{k-1})=\Hy^k\times \Sf^{n-k+1}$, it follows that $\Theta(\mathcal{S}\cap (\R^{n+1}\setminus\R^{k-1}))=(\Hy^k\times \Sf^{n-k+1})\cap V^\perp$. \vspace{1ex}\qed
 
  The next result reduces the problem of deciding when two  substantial  umbilical submanifolds  of $(\R^{n+1}\setminus\R^{k-1},g)=\Hy^k\times \Sf^{n-k+1}$ with the same dimension are congruent to an easier one in terms of their corresponding space-like subspaces.

\begin{proposition}\label{congruence} Let $S$ and $\tilde S$ be  umbilical submanifolds  of $(\R^{n+1}\setminus\R^{k-1},g)$ with the same dimension. Let  $V$ and $\tilde V$ be their correspondent space-like subspaces of $\Le^{n+3}=W_1\oplus W_2$.
Then $S$ and $\tilde S$ are congruent in  $(\R^{n+1}\setminus\R^{k-1},g)$ if and only if  there exists $T\in O_1(n+3)$ leaving $W_1$ and $W_2$ invariant such that $T(V)=\tilde V$.
\end{proposition}
\proof Let $S$ and $\tilde S$ be  umbilical submanifolds  of $(\R^{n+1}\setminus\R^{k-1},g)$ with the same dimension, and let  $V$ and $\tilde V$ be their correspondent space-like subspaces of $\Le^{n+3}$. By Proposition \ref{intersection},  $\Theta(S)=(\Hy^k\times \Sf^{n-k+1})\cap V^\perp$ and $\Theta(\tilde{S})=(\Hy^k\times \Sf^{n-k+1})\cap {\tilde V}^\perp$. Since any isometry $F\colon (\R^{n+1}\setminus\R^{k-1},g)\to (\R^{n+1}\setminus\R^{k-1},g)$ is given by  $\Theta\circ F=T\circ \Theta$ for some $T\in O_1(n+3)$ that
leaves $W_1$ and $W_2$ invariant by Theorem \ref{isometries}, it follows that
$$\begin{array}{l}
F(S)=\tilde S\Longleftrightarrow \Theta(F(S))=\Theta(\tilde S)
\Longleftrightarrow T(\Theta(S))=\Theta(\tilde S)\vspace{1ex}\\
\Longleftrightarrow T((\Hy^k\times \Sf^{n-k+1})\cap V^\perp)=(\Hy^k\times \Sf^{n-k+1})\cap {\tilde V}^\perp\Longleftrightarrow T(V^\perp)={\tilde V}^\perp.\end{array} \qed
$$

\begin{proposition}\label{alg} Let $W$ be a vector space endowed with a Lorentzian inner product and let $W=W_1\oplus W_2$ be an orthogonal direct sum decomposition, with $W_1$ time-like.  For $w\in W$, write $w=w^T+w^\perp$ according to the decomposition $W=W_1\oplus W_2$. Let $V$ and $\tilde V$  be space-like subspaces of $W$ with the same dimension. Assume that there exists a  linear isometry $T\colon W\to W$ preserving $W_1$ and $W_2$ such that $T(V)=\tilde V$. Then, for any orthonormal bases 
 $\{\xi_1, \ldots, \xi_m\}$ and  $\{\tilde{\xi}_1, \ldots, \tilde{\xi}_m\}$ of  $V$ and $\tilde V$, respectively, the Gram matrices $(\<\xi^T_i, \xi^T_j\>)_{1\leq i,j\leq m}$ 
and $(\<\tilde{\xi}^T_i, \tilde{\xi}^T_j\>)_{1\leq i,j\leq m}$ (or equivalently, the Gram matrices $(\<\xi^\perp_i, \xi^\perp_j\>)_{1\leq i,j\leq m}$ and  $(\<\tilde{\xi}^\perp_i, \tilde{\xi}^\perp_j\>)_{1\leq i,j\leq m}$), have the same eigenvalues, counted with multiplicities.  
 
 Conversely, assume that there exist orthonormal bases 
 $\{\xi_1, \ldots, \xi_m\}$ and  $\{\tilde{\xi}_1, \ldots, \tilde{\xi}_m\}$ of  $V$ and $\tilde V$, respectively, such that 
 \begin{itemize}
 \item[(i)] the subspaces spanned by  $\{\xi^T_1, \ldots, \xi^T_m\}$ and  $\{\tilde{\xi}^T_1, \ldots, \tilde{\xi}^T_m\}$  either are both degenerate and have dimension $m$, or are both nondegenerate; 
\item[(ii)] the Gram matrices $(\<\xi^T_i, \xi^T_j\>)_{1\leq i,j\leq m}$ and  $(\<\tilde{\xi}^T_i, \tilde{\xi}^T_j\>)_{1\leq i,j\leq m}$ 
(or equivalently, the Gram matrices $(\<\xi^\perp_i, \xi^\perp_j\>)_{1\leq i,j\leq m}$ and  $(\<\tilde{\xi}^\perp_i, \tilde{\xi}^\perp_j\>)_{1\leq i,j\leq m}$) 
have the same eingenvalues, counted with multiplicities.
\end{itemize}
Then there is a  linear isometry $T\colon W\to W$ preserving $W_1$ and $W_2$ such that $T(V)=\tilde V$.
\end{proposition}

 For the proof of Proposition \ref{alg} we need the following lemma.
 
 \begin{lemma}\label{le:gram} Let $W$ be a vector space endowed with a Lorentzian  inner product $\<\;,\;\>^\sim$, and let $\{\xi_1, \ldots, \xi_m\},  \{\tilde{\xi}_1, \ldots, \tilde{\xi}_m\}\subset W$. Assume that the subspaces spanned by $\{\xi_1, \ldots, \xi_m\}$ and $\{\tilde{\xi}_1, \ldots, \tilde{\xi}_m\}$ either are both degenerate and have dimension $m$ or are both nondegenerate. Suppose further that $\<\xi_i, \xi_j\>^\sim=\<\tilde{\xi}_i, \tilde{\xi}_j\>^\sim$ for all $1\leq i,j\leq m$. Then there exists a linear isometry $T\colon W\to W$ such that $T\xi_i=\tilde{\xi}_i$, $1\leq i\leq m$.
\end{lemma}
\proof Define $A\colon \R^m\to W$ and  $B\colon \R^m\to W$ by $Ae_i=\xi_i$ and $Be_i=\tilde{\xi}_i$, $1\leq i\leq m$, where $\{e_1, \ldots, e_m\}$ is the canonical basis of $\R^m$.  Then $A(\R^m)=\spa\{\xi_1, \ldots, \xi_m\}$ and $B(\R^m)=\spa \{\tilde{\xi}_1, \ldots, \tilde{\xi}_m\}$. Let $A^t\colon W\to \R^m$ and  $B^t\colon W\to \R^m$ be the transposes of $A$ and $B$, that is, $\<A^tw,v\>=\<w,Av\>^\sim$
for all $v\in \R^m$ and $w\in W$, where $\<\;, \;\>$ is the standard inner product of $\R^m$, and $B^t$ is similarly defined. Then 
$$\<A^tAe_i,e_j\>=\<Ae_i, Ae_j\>^\sim=\<\xi_i, \xi_j\>^\sim=\<\tilde{\xi}_i, \tilde{\xi}_j\>^\sim=\<Be_i,Be_j\>^\sim=\<B^tBe_i,e_j\>$$ for all $1\leq i,j\leq m$. Thus  $A^tA=B^tB$.
Since $A^tA\colon \R^m\to \R^m$ is a symmetric operator, there exists an orthonormal basis $v_1, \ldots, v_m$ of $\R^m$ and $\lambda_1, \ldots, \lambda_m\in \R$ such that 
$A^tAv_i=\lambda_i v_i$, $1\leq i\leq m$, that is, 
 $$\<Av_i, Av_j\>=\lambda_i\delta_{ij}=\<Bv_i, Bv_j\>,\;\;\;1\leq i, j\leq m.$$
 If both $A(\R^m)$ and $B(\R^m)$ are nondegenerate, this implies that there exists a linear isometry $T\colon W\to W$ such that $TAv_i=Bv_i$, $1\leq i\leq m$, and hence $TA=B$.
 On the other hand, if both $A(\R^m)$ and $B(\R^m)$ are degenerate and have dimension $m$, 
 there exists exactly one $i\in \{1,\ldots, m\}$ such that $\lambda_i=0$, in which case both $Av_i$ and $Bv_i$ are nonzero light-like vectors. Also in this case, there exists a linear isometry $T\colon W\to W$ such that $TAv_i=Bv_i$, $1\leq i\leq m$, and hence $TA=B$. 
 Thus  $T\xi_i=TAe_i=Be_i=\tilde{\xi}_i$, $1\leq i\leq m$.   \vspace{2ex}\qed

\noindent \emph{Proof of Proposition \ref{alg}:} Assume first that $T\colon W\to W$ is a  linear isometry preserving $W_1$ and $W_2$ such that $T(V)=\tilde V$, and let $\{\xi_1, \ldots, \xi_m\}$ and  $\{\tilde{\xi}_1, \ldots, \tilde{\xi}_m\}$ be orthonormal bases of  $V$ and $\tilde V$, respectively. Then 
$T\xi_j=\sum_{i=1}^m a_{ij}\tilde{\xi}_i$, $1\leq j\leq m$, for some $A=(a_{ij})\in O(m)$.  Since $T$ preserves $W_1$ and $W_2$, then $T\xi^T_j=\sum_{i=1}^m a_{ij}\tilde{\xi}^T_i$,  and hence the Gram matrices $G=(\<\xi^T_i, \xi^T_j\>)_{1\leq i,j\leq m}$ and  $\tilde G=(\<\tilde{\xi}^T_i, \tilde{\xi}^T_j\>)_{1\leq i,j\leq m}$ are related by $G=A\tilde{G}A^t$. It follows that  $G$ and $\tilde G$  have the same eingenvalues, counted with multiplicities.  

Conversely, assume that there exist orthonormal bases 
 $\{\xi_1, \ldots, \xi_m\}$ and  $\{\tilde{\xi}_1, \ldots, \tilde{\xi}_m\}$ of  $V$ and $\tilde V$, respectively, satisfying conditions $(i)$ and $(ii)$ in the statement.  
Condition $(ii)$ implies that there exists $A=(a_{ij})\in O(m)$ such that $G=A\tilde{G}A^t$. Defining 
$\bar{\xi}_j=\sum_{i=1}^m a_{ij}\tilde{\xi}^T_i$, the Gram matrices $\bar G=(\<\bar{\xi}_i, \bar{\xi}_j\>)_{1\leq i,j\leq m}$ and $\tilde G$  
are also related by $\bar G=A\tilde{G}A^t$, and hence $G=\bar G$. By Lemma \ref{le:gram} and condition $(i)$,  there exists a linear isometry $T_1\colon W_1\to W_1$ such that $T_1\xi^T_j=\bar{\xi}_j$, $1\leq j\leq m$. 

  Now define $\hat{\xi}_j=\sum_{i=1}^m a_{ij}\tilde{\xi}^\perp_i$. We have
  \begin{eqnarray*} \<\hat{\xi}_j, \hat{\xi}_k\>&=&\<\sum_{i=1}^m a_{ij}\tilde{\xi}^\perp_i, 
  \sum_{\ell=1}^m a_{\ell k}\tilde{\xi}^\perp_\ell\>\vspace{1ex}\\
&=&  \sum_{i,\ell=1}^m a_{ij}a_{\ell k}\<\tilde{\xi}^\perp_i, \tilde{\xi}^\perp_\ell\>\vspace{1ex}\\
&=&  \sum_{i,\ell=1}^m a_{ij}a_{\ell k}(\delta_{i\ell}-\<\tilde{\xi}^T_i, \tilde{\xi}^T_\ell\>)\vspace{1ex}\\
&=& \sum_{i=1}^m a_{ij}a_{ik} -\sum_{i,\ell=1}^m a_{ij}a_{\ell k}\<\tilde{\xi}^T_i, \tilde{\xi}^T_\ell\>\vspace{1ex}\\
&=& \delta_{jk}-\<{\xi}^T_j, {\xi}^T_k\>\vspace{1ex}\\
&=&\<{\xi}^\perp_j, {\xi}^\perp_k\>.
\end{eqnarray*}
Since $W_2$ is space-like, there exists  a linear isometry $T_2\colon W_2\to W_2$ such that $T_2\xi^\perp_j=\hat{\xi}_j$, $1\leq j\leq m$. Defining
$T\colon W\to W$ by $T|_{W_1}=T_1$ and $T|_{W_2}=T_2$, it follows that $T$ is a linear isometry preserving $W_1$ and $W_2$ such that 
$$T\xi_j=T_1\xi^T_j+T_2\xi^\perp_j=\sum_{i=1}^m a_{ij}\tilde{\xi}^T_i+\sum_{i=1}^m a_{ij}\tilde{\xi}^\perp_i=\sum_{i=1}^m a_{ij}\tilde{\xi}_i,
$$
and hence  $T(V)=\tilde{V}$. \vspace{1ex}\qed
\begin{remark} \emph{In Lemma \ref{le:gram}, if the subsets $\{\xi_1, \ldots, \xi_m\}$ and $\{\tilde{\xi}_1, \ldots, \tilde{\xi}_m\}$ are assumed to be linearly independent, then the hypothesis that $\<\xi_i, \xi_j\>^\sim=\<\tilde{\xi}_i, \tilde{\xi}_j\>^\sim$ for all $1\leq i,j\leq m$ already implies that the subspaces spanned by $\{\xi_1, \ldots, \xi_m\}$ and $\{\tilde{\xi}_1, \ldots, \tilde{\xi}_m\}$ either are both degenerate or  both nondegenerate, for the degeneracy of the subspace spanned by a linearly independent subset $\{\xi_1, \ldots, \xi_m\}$ is characterized by the vanishing of the determinant of the Gram matrix  $(\<\xi_i, \xi_j\>)_{1\leq i,j\leq m}$. This follows from the proof of Lemma~\ref{le:gram}, which shows that if $A(\R^m)$ has dimension $m$ then it is degenerate if and only if one of the eigenvalues $\lambda_1, \ldots, \lambda_m$ of the operator $A^tA$ is zero. Accordingly, if in the converse statement of Proposition \ref{alg} the subsets $\{\xi^T_1, \ldots, \xi^T_m\}$ and $\{\tilde{\xi}^T_1, \ldots, \tilde{\xi}^T_m\}$ are assumed to be linearly independent, then condition $(i)$ is a consequence of condition $(ii)$.}
\end{remark}

 The next result gives an algebraic criterion to decide whether an umbilical submanifold  of  $(\R^{n+1}\setminus\R^{k-1},g)$ is substantial.

\begin{proposition} \label{prop:substantial}  An umbilical submanifold $S$ of  $(\R^{n+1}\setminus\R^{k-1},g)$ with corresponding space-like subspace $V\subset \Le^{n+3}$ is substantial if and only if  $V\cap W_1$ and $V\cap W_2$ are trivial. 
\end{proposition}
\proof Since $\Psi(S)=\Ee^{n+1}\cap V^\perp$, then $S$ is contained in a totally geodesic submanifold $\tilde S$, with $\Psi(\tilde S)=\Ee^{n+1}\cap {\tilde V}^\perp$ for a space-like subspace $\tilde V\subset \Le^{n+3}$, if $V\supset \tilde V$. By Corollary~ \ref{cor:totgeosub}, 
$\tilde V=\tilde{V}_1\oplus\tilde{V}_2$ with 
$\tilde{V}_1\subset W_1$ and $\tilde{V}_2\subset W_2$.
Therefore, $S$ is contained in some totally geodesic submanifold $\tilde S$ if and only if either $V\cap W_1$ or $V\cap W_2$ is non trivial. \qed

\begin{corollary}\label{substantial} The substantial codimension of an  umbilical submanifold  of $\Hy^k\times \Sf^{n-k+1}$ is at most $\mbox{min}\,\{k+1, n-k+2\}$.
 \end{corollary}
 
 We are now able to determine how many congruence classes of 
 umbilical submanifolds of $\Hy^k\times \Sf^{n-k+1}$  with 
 a given substantial codimension are there.

\begin{theorem}\label{thm:congclasses} There exists precisely a $p$-parameter family of congruence classes of umbilical submanifolds of $(\R^{n+1}\setminus \R^{k-1}, g)=\Hy^k\times \Sf^{n-k+1}$ with substantial codimension~$p$. 
 \end{theorem} 
\proof By Proposition \ref{congruence}, two umbilical submanifolds $S$ and $\tilde{S}$ of 
 $(\R^{n+1}\setminus \R^{k-1}, g)$  are congruent in $(\R^{n+1}\setminus \R^{k-1}, g)$ if and only if  there exists $T\in O_1(n+3)$ leaving $W_1$ and $W_2$ invariant such that 
$T(V)=\tilde V$, where $V$ and $\tilde V$ are the space-like subspaces correspondent to $S$ and $\tilde{S}$, respectively.  
 
 Let $\{\xi_1, \ldots, \xi_p\}$ and  $\{\tilde{\xi}_1, \ldots, \tilde{\xi}_p\}$ be orthonormal bases of  $V$ and $\tilde V$, respectively. Since $V$ and $\tilde V$ are substantial, by Proposition \ref{prop:substantial} the subspaces spanned by  $\{\xi^T_1, \ldots, \xi^T_m\}$ and  $\{\tilde{\xi}^T_1, \ldots, \tilde{\xi}^T_m\}$ have dimension $p$.
  It follows from Proposition \ref{alg} that there exists $T\in O_1(n+3)$ leaving $W_1$ and $W_2$ invariant such that $T(V)=\tilde{V}$ if and only if the Gram matrices 
$(\<\xi^T_i, \xi^T_j\>)_{1\leq i,j\leq m}$ and  
$(\<\tilde{\xi}^T_i, \tilde{\xi}^T_j\>)_{1\leq i,j\leq m}$ have the same eingenvalues, counted with multiplicities. Therefore, the congruence classes of  umbilical submanifolds of $(\R^{n+1}\setminus \R^{k-1}, g)$ with  substantial codimension $p$ are in one-to-one correspondence with the sets of eigenvalues of such  Gram matrices. \qed

\subsection{The hypersurface case}.

The next result gives necessary and sufficient conditions for two umbilical hypersurfaces of 
$(\R^{n+1}\setminus \R^{k-1}, g)$ to be congruent in terms of their Euclidean data.

\begin{proposition}\label{umb3a} The following assertions hold:
\begin{itemize}
\item[(i)]
 The hyperspheres $\Sf^n(x_0;r)$ and $\Sf^n(\tilde{x}_0;\tilde{r})$ are congruent in $(\R^{n+1}\setminus \R^{k-1}, g)$ if and only if  ${\displaystyle \frac{\|{x}_0^\perp\|}{r}=\frac{\|\tilde{x}_0^\perp\|}{\tilde{r}}}$.
\item[(ii)] The hypersphere $\Sf^n(x_0;r)$ and the affine hyperplane $\mathcal{H}(N;c)$  are congruent in $(\R^{n+1}\setminus \R^{k-1}, g)$ if and only if  ${\displaystyle \frac{\|{x}_0^\perp\|}{r}=\|N^\perp\|}$ and $c\neq 0$ if $N^T=0$. 
\item[(iii)] Two affine hyperplanes $\mathcal{H}(N;c)$ and $\mathcal{H}(\tilde{N};\tilde{c})$ are congruent in $(\R^{n+1}\setminus \R^{k-1}, g)$ if and only if  $\|N^\perp\|=\|\tilde{N}^\perp\|$ and either $c\neq 0\neq \tilde c$ or $c=0=\tilde c$ if $N^T$ (and hence $\tilde{N}^T$) vanishes. 
\end{itemize}
\end{proposition} 
\proof $(i)$ Let  ${\displaystyle \zeta=\frac{1}{r}\Psi(x_0)+\frac{r}{2}w}$ and ${\displaystyle \tilde\zeta=\frac{1}{\tilde{r}}\Psi(\tilde{x}_0)+\frac{\tilde{r}}{2}w}$ be the unit space-like vectors associated with $\Sf^n(x_0;r)$ and $\Sf^n(\tilde{x}_0;\tilde{r})$, respectively. Write
$x_0=x_0^T+x_0^\perp$, with  $x_0^T\in \R^{k-1}$ and  $x_0^\perp\in \R^{n-k+2}=(\R^{k-1})^\perp$, and decompose $\tilde{x}_0$ accordingly. Then $\zeta=\zeta^T+\zeta^\perp$, with 
$$\zeta^T=\frac{1}{r}v +\frac{1}{2}
\left(r-\frac{\|x_0\|^2}{r}\right)w+\frac{1}{r}Cx_0^T\,\,\,\mbox{and}\,\,\,\zeta^\perp=\frac{1}{r}Cx_0^\perp,$$ is the decomposition of $\zeta$ with respect to the orthogonal decomposition $\Le^{n+3}=W_1\oplus W_2$, with $W_1=V_1\oplus\spa\{v,w\}$ and $W_2=V_2$. A similar decomposition holds for $\tilde{\zeta}$.  

By Proposition \ref{congruence}, $\Sf^n(x_0;r)$ and $\Sf^n(\tilde{x}_0;\tilde{r})$ are congruent in $(\R^{n+1}\setminus \R^{k-1}, g)$ if and only if  there exists $T\in O_1(n+3)$ leaving $W_1$ and $W_2$ invariant such that $T(\zeta)=\tilde \zeta$. It follows from Proposition \ref{alg} that this is the case if and only if either  $\zeta^T$ and $\tilde{\zeta}^T$ are nonzero light-like vectors  or if  neither of them is light-like and $\<\zeta^T,\zeta^T\>=\<\tilde{\zeta}^T, \tilde{\zeta}^T\>$ (or equivalently, $\<\zeta^\perp,\zeta^\perp\>=\<\tilde{\zeta}^\perp, \tilde{\zeta}^\perp\>$), that is, if and only if ${\displaystyle \frac{\|{x}_0^\perp\|}{r}=\frac{\|\tilde{x}_0^\perp\|}{\tilde{r}}}$.

\noindent $(ii)$ Let $\eta=CN-cw$ be the unit space-like vector representing $\mathcal{H}(N;c)$. By Proposition~\ref{congruence}, $\Sf^n(x_0;r)$ and $\mathcal{H}(N;c)$ are congruent in $(\R^{n+1}\setminus \R^{k-1}, g)$ if and only if  there exists $T\in O_1(n+3)$ leaving $W_1$ and $W_2$ invariant such that $T(\zeta)=\eta$. It follows from Proposition \ref{alg} that this is the case if and only if either $\zeta^T$ and $\eta^T$ are nonzero light-like vectors or neither of them is light-like, and $\<\zeta^\perp,\zeta^\perp\>=\<\eta^\perp, \eta^\perp\>$. This is the case  if and only if ${\displaystyle \frac{\|{x}_0^\perp\|}{r}=\|N^\perp\|}$ and $c\neq 0$ if $N^T=0$.\vspace{1ex}

\noindent $(iii)$ Follows in a similar way as in $(i)$ and $(ii)$. \qed

\begin{corollary}\label{cor:hypcase} There exists precisely a one-parameter family of umbilical non totally geodesic hypersurfaces of $(\R^{n+1}\setminus \R^{k-1}, g)=\Hy^k\times \Sf^{n-k+1}$. More precisely,
for a fixed $x_0\in \R^{n+1}$ with $x_0^\perp\neq 0$, the set $\{\Sf^n(x_0;r), r>0\}$ contains exactly one representative of each congruence class of such hypersurfaces.
\end{corollary}
\proof By part $(i)$ of Proposition \ref{umb3a}, $\Sf^n(x_0;r)$ is not congruent to $\Sf^n({x}_0;\tilde{r})$ if $r\neq \tilde r$. On the other hand,  any $\Sf^n(\tilde{x}_0;\tilde{r})$ with $\tilde{x}_0^\perp\neq 0$ is congruent to $\Sf^n(x_0;r)$  for ${\displaystyle r=\frac{\tilde{r}\|x_0^\perp\|}{\|\tilde{x}_0^\perp\|}}$, and any $\mathcal{H}(N;c)$ with $N^\perp\neq 0$ and with $c\neq 0$ if $N^T=0$ is congruent to $\Sf^n(x_0;r)$  for ${\displaystyle r=\frac{\|x_0^\perp\|}{\|N^\perp\|}}$. \qed

\subsection{The codimension two case}

In this section we solve the congruence problem for umbilical submanifolds with codimension two of  $(\R^{n+1}\setminus \R^{k-1}, g)$. In particular, this implies a classification of (the congruence classes of) all umbilical  submanifolds  of $\Sf^n\times \R$ and $\Hy^n\times \R$.

Notice that any sphere of $\R^{n+1}$ with codimension two can be produced (in a unique way) as the orthogonal intersection of a hypersphere $\Sf^n(x_0;r)$ and an affine hyperplane $\mathcal{H}(N;c)$. Also, any affine subspace  of $\R^{n+1}$ with codimension two whose distance to the origin is $c>0$ can be produced (in a unique way) as the orthogonal intersection of an affine hyperplane $\mathcal{H}(N_1;c)$ and a hyperplane 
$\mathcal{H}(N_2;0)$. For a linear subspace with codimension $2$, however, the representation $\mathcal{H}(N_1,0)\cap \mathcal{H}(N_2,0)$  is not unique. Notice also that,  if in the latter representation the vectors  $N_1^T$ and $N_2^T$ are linearly dependent, then one can replace it by another such representation for which $N_2^T=0$.

\begin{proposition}\label{umb1} 
$(a)$ The spheres $\Sf^n(x_0;r)\cap \mathcal{H}(N;c)$ and $\Sf^n(\tilde{x}_0;\tilde{r})\cap \mathcal{H}(\tilde{N};\tilde{c})$ are congruent in $(\R^{n+1}\setminus \R^{k-1}, g)$ if and only if one of the following holds:
\begin{itemize}

\item[(i)] $(N^T,c)=(0,0)=(\tilde{N}^T,\tilde{c})$ and $r^{-1}\|x_0^\perp\|=\tilde{r}^{-1}\|\tilde{x}_0^\perp\|$; or
\item[(ii)] $(N^T,c)\neq(0,0)\neq(\tilde{N}^T,\tilde{c})$ and
\begin{equation}\label{eq:congcond} \left\{\begin{array}{l}
{\displaystyle \frac{\|x_0^\perp\|^2}{r^2}+\|N^\perp\|^2= \frac{\|\tilde{x}_0^\perp\|^2}{\tilde{r}^2}+\|\tilde{N}^\perp\|^2}\vspace{1ex}\\
{\displaystyle \frac{1}{r^2}(\|x_0^\perp\|^2\|N^\perp\|^2-\<x_0^\perp, N^\perp\>^2)=\frac{1}{\tilde{r}^2}(\|\tilde{x}_0^\perp\|^2\|\tilde{N}^\perp\|^2-\<\tilde{x}_0^\perp, \tilde{N}^\perp\>^2}).\end{array} \right. 
\end{equation}

\end{itemize}

$(b)$ The sphere $\Sf^n(x_0;r)\cap \mathcal{H}(N;c)$ and the affine subspace $\mathcal{H}(N_1;c_1)\cap\mathcal{H}(N_2;0)$ are congruent in $(\R^{n+1}\setminus \R^{k-1}, g)$ if and only if 
one of the following holds:
\begin{itemize}
    \item[(i)] $(N^T,c)=(0,0)$, $N_2^T=0$,  $r^{-1}\|x_0^\perp\|=\|N_1^\perp\|$ and $c_1\neq 0$ if $N_1^T=0$; or
    \item[(ii)] $(N^T,c)\neq(0,0)$, $N_2^T\neq 0$, with $\{N_1^T,N_2^T\}$ linearly independent if $c_1=0$, and
    \begin{equation}\label{eq:congcondb} \left\{\begin{array}{l}
{\displaystyle \frac{\|x_0^\perp\|^2}{r^2}+\|N^\perp\|^2= \|N^\perp_1\|^2+\|N^\perp_2\|^2}\vspace{1ex}\\
{\displaystyle \|x_0^\perp\|^2\|N^\perp\|^2-\<x_0^\perp, N^\perp\>^2=
r^2(\|N^\perp_1\|^2\|N^\perp_2\|^2-\<N^\perp_1,N^\perp_2\>^2).}\end{array} \right.
\end{equation}
\end{itemize} 

$(c)$ The affine subspaces $\mathcal{H}(N_1;c)\cap\mathcal{H}(N_2;0)$ and $\mathcal{H}(\tilde{N}_1;\tilde c)\cap\mathcal{H}(\tilde{N}_2;0)$ are congruent in $(\R^{n+1}\setminus \R^{k-1}, g)$ if and only if 
one of the following holds:
\begin{itemize}
    \item[(i)] $N_2^T=0=\tilde{N}_2^T$, $\|N_1^\perp\|=\|\tilde{N}_1^\perp\|$ and either $c\neq 0\neq \tilde c$ or $c=0=\tilde c$ if $N_1^T$ (and hence $\tilde{N}_1^T$) vanishes; or
    \item[(ii)] $N_2^T\neq0\neq\tilde{N}_2^T$, with $\{N_1^T,N_2^T\}$ linearly independent if $c=0$ and $\{\tilde{N}_1^T,\tilde{N}_2^T\}$ linearly independent if $\tilde{c}=0$, and
    \begin{equation}\label{eq:congcondc} \left\{\begin{array}{l}
{\displaystyle \|N^\perp_1\|^2+\|N^\perp_2\|^2= \|\tilde{N}^\perp_1\|^2+\|\tilde{N}^\perp_2\|^2}\vspace{1ex}\\
{\displaystyle \|N^\perp_1\|^2\|N^\perp_2\|^2-\<N^\perp_1,N^\perp_2\>^2=
\|\tilde{N}^\perp_1\|^2\|\tilde{N}^\perp_2\|^2-\<\tilde{N}^\perp_1,\tilde{N}^\perp_2\>^2.}\end{array} \right.
\end{equation}
\end{itemize}
\end{proposition}
\proof The spheres $\Sf^n(x_0;r)\cap \mathcal{H}(N;c)$ and $\Sf^n(\tilde{x}_0;\tilde{r})\cap \mathcal{H}(\tilde{N};\tilde{c})$ are represented by the two-dimensional space-like subspaces $V=\spa\{\xi_1, \xi_2\}$ and $\tilde{V}=\spa\{\tilde{\xi}_1, \tilde{\xi}_2\}$, where 
$$
\xi_1=\frac{1}{r}\Psi(x_0)+\frac{r}{2}w=\frac{1}{r}v+\frac{1}{2}
\left(r-\frac{\|x_0\|^2}{r}\right)w+\frac{1}{r}Cx_0,\;\;\;\;\xi_2=-cw+CN,
$$
and $\tilde{\xi}_1$,  $\tilde{\xi}_2$ are given in a similar way. 
Thus
$$\xi_1^T=\frac{1}{r}v +\frac{1}{2}
\left(r-\frac{\|x_0\|^2}{r}\right)w+\frac{1}{r}Cx_0^T,\,\,\,\,\,\,\,\,\xi_1^\perp=\frac{1}{r}Cx_0^\perp$$
and
$$\xi_2^T=-cw+CN^T,\,\,\,\,\,\,\,\,\,\xi_2^\perp =CN^\perp$$
are the decompositions of $\xi_1$ and $\xi_2$, respectively,  with respect to the orthogonal decomposition $\Le^{n+3}=W_1\oplus W_2$, with similar decompositions for $\tilde{\xi}_1$ and $\tilde{\xi}_2$. The orthogonality condition between $\Sf^n(x_0;r)$ and $\mathcal{H}(N;c)$ (respectively,
 $\Sf^n(\tilde{x}_0;\tilde{r})$ and $\mathcal{H}(\tilde{N};\tilde{c})$) is that $\<N, x_0\>=c$ (respectively, $\<\tilde{N}, \tilde{x}_0\>=\tilde{c}$).
 
 By Proposition \ref{congruence}, the spheres $\Sf^n(x_0;r)\cap \mathcal{H}(N;c)$ and $\Sf^n(\tilde{x}_0;\tilde{r})\cap \mathcal{H}(\tilde{N};\tilde{c})$ are congruent in $(\R^{n+1}\setminus \R^{k-1}, g)$ if and only if  there exists $T\in O_1(n+3)$ leaving $W_1$ and $W_2$ invariant such that $T(V)=\tilde V$.  By Proposition \ref{alg}, if $\dim\spa\{\xi_1^T,\xi_2^T\}=\spa\{\tilde{\xi}_1^T,\tilde{\xi}_2^T\}=2$, this is the case 
if and only if the Gram matrices $(\<\xi^T_i, \xi^T_j\>)_{1\leq i,j\leq 2}$ and  $(\<\tilde{\xi}^T_i, \tilde{\xi}^T_j\>)_{1\leq i,j\leq 2}$ (or equivalently, the Gram matrices $(\<\xi^\perp_i, \xi^\perp_j\>)_{1\leq i,j\leq 2}$ and  $(\<\tilde{\xi}^\perp_i, \tilde{\xi}^\perp_j\>)_{1\leq i,j\leq 2})$ have the same eigenvalues, counted with multiplicities.
Since these Gram matrices have order two, they have the same eigenvalues, counted with multiplicities, if and only if they have equal traces and determinants.
The latter conditions are equivalent to (\ref{eq:congcond}). 
Notice that  $\dim\spa\{\xi_1^T,\xi_2^T\}=1$ if and only if  $(N^T,c)=(0,0)$. If this is the case, then there exists $T\in O_1(n+3)$ leaving $W_1$ and $W_2$ invariant such that $T(V)=\tilde V$ if and only if also $(\tilde{N}^T,\tilde c)=(0,0)$ and there exists $T_1\in O_1(k+1)$ such that $T_1\xi_1^T=\tilde{\xi}_1^T$. The latter holds if and only if $\<\xi_1^T,\xi_1^T\>=\<\tilde{\xi}_1^T,\tilde{\xi}_1^T\>$, that is, if and only if $\displaystyle{\frac{x_0^\perp}{r}=\frac{\tilde{x}_0^\perp}{\tilde r}}$ The proofs of $(ii)$ and $(iii)$ are similar. \qed

\begin{corollary}\label{umb3} $a)$ The spheres $\Sf^n(x_0;r)\cap \mathcal{H}(N;c)$ and $\Sf^n(\tilde{x}_0;\tilde{r})\cap \mathcal{H}(\tilde{N};\tilde{c})$ are congruent in $(\R^{n+1}\setminus \{0\}, g)$ if and only if ${\displaystyle \frac{\|\tilde{x}_0\|}{\tilde r}= \frac{\|{x}_0\|}{r}\;\;\;\mbox{and}\;\;\;\frac{\tilde{c}^2}{\tilde{r}^2}=\frac{c^2}{r^2}.}$\vspace{1ex}\\
$b)$ The sphere $\Sf^n(x_0;r)\cap \mathcal{H}(N;c)$ and the affine subspace $\mathcal{H}(N_1;c_1)\cap\mathcal{H}(N_2;0)$ are congruent in $(\R^{n+1}\setminus \{0\}, g)$  if and only if $c=0$, $\|x_0\|=r$ and $c_1\neq 0$.\vspace{1ex}\\
$c)$ Two affine subspaces $\mathcal{H}(N_1;c)\cap\mathcal{H}(N_2;0)$ and $\mathcal{H}(\tilde{N}_1;\tilde c)\cap\mathcal{H}(\tilde{N}_2;0)$ are congruent in $(\R^{n+1}\setminus \{0\}, g)$  if and only if either $c=0=\tilde c$ or $c\neq 0 \neq \tilde c$.
\end{corollary}
 
The preceding corollary allows us to recover the classification in \cite{mt1} of the umbilical submanifolds of  $\mathbb{S}^n\times \mathbb{R}$.

\begin{corollary} \label{cor:cod2snr}
For any $x_0\in \R^{n+1}\setminus \{0\}$ and any unit vector $N\in \R^{n+1}$ with $\<x_0, N\>=0$, the two-parameter family
$\{ \Sf^n(x_c;r)\cap \mathcal{H}(N;c),\;x_c=x_0+c N, c\geq 0, r>0\}$, together with the one-parameter family  $\{ \Sf^n(x_0;r)\cap \mathcal{H}(x_0/\|x_0\|;\|x_0\|),\; r>0\}$, contain precisely one representative of each congruence class of umbilical non totally geodesic submanifolds of codimension two of $(\R^{n+1}\setminus \{0\}, g)=\Sf^n\times \R$. The sphere $\Sf^n(x_c;r)\cap \mathcal{H}(N;c)$ has substantial codimension two if and only if  $c>0$.
\end{corollary}
\proof By Proposition \ref{umb3}, if $\Sf^n(x_c;r)\cap \mathcal{H}(N;c)$ and $ \Sf^n(x_{\tilde{c}};\tilde{r})\cap \mathcal{H}(\tilde{N};\tilde{c})$, with $c,\tilde{c}>0$ and $r,\tilde{r}>0$, are congruent, then ${\displaystyle \frac{\|{x}_{\tilde{c}}\|}{\|{x}_c\|}=\frac{\tilde{r}}{r}=\frac{\tilde{c}}{c}}$. Thus
${\displaystyle \frac{\tilde{c}^2+\|x_0\|^2}{{c}^2+\|x_0\|^2}=\frac{\|{x}_{\tilde{c}}\|^2}{\|{x}_c\|^2}=\frac{\tilde{c}^2}{c^2}}$,
which implies that $\tilde{c}=c$, and hence $\tilde{r}=r$. 

Similarly, Proposition \ref{umb3} also implies that $\Sf^n(x_0;r)\cap \mathcal{H}(x_0/\|x_0\|;\|x_0\|)$ and $\Sf^n(x_0;\tilde r)\cap \mathcal{H}(x_0/\|x_0\|;\|x_0\|)$ are congruent if and only if $r=\tilde r$.

 Now, given any $\Sf^n(\tilde{x}_0;\tilde{r})\cap \mathcal{H}(\tilde{N};\tilde{c})$ with
 $\|\tilde{x}_0\|\neq 0$, set ${\displaystyle \lambda^2=\|\tilde{x}_0\|^2}/{\tilde{c}^2}$.
 Notice that $\lambda\geq 1$ by the Cauchy-Schwarz inequality, for $\<\tilde N, \tilde{x}_0\>=\tilde c$. 
If $\lambda^2=1$, that is,  $\tilde{c}^2=\|\tilde{x}_0\|^2$, then 
 $\Sf^n(\tilde{x}_0;\tilde{r})\cap \mathcal{H}(\tilde{N};\tilde{c})$ is congruent to 
 $\Sf^n(x_0;r)\cap \mathcal{H}(x_0/\|x_0\|;\|x_0\|)$ for ${\displaystyle r=\frac{\tilde r\|x_0\|}{\|\tilde{x}_0\|}}$.
  
  If $\lambda^2>1$, choose ${\displaystyle c=\sqrt{\frac{\|x_0\|^2}{\lambda^2-1}}}$ and ${\displaystyle r=\frac{c\tilde{r}}{\tilde{c}}}$. Then  
  ${\displaystyle \frac{\|\tilde{x}_0\|}{\|x_c\|}=\frac{\tilde{r}}{r}=\frac{\tilde{c}}{c},}$
 hence $\Sf^n(\tilde{x}_0;\tilde{r})\cap \mathcal{H}(\tilde{N};\tilde{c})$ is congruent to 
$\Sf^n(x_c;r)\cap \mathcal{H}(N;c)$ by Proposition \ref{umb3}. 

  Finally, any $\mathcal{H}(N_1;c_1)\cap\mathcal{H}(N_2;0)$ with $c_1\neq 0$ is congruent to 
$\Sf^n(x_c;r)\cap \mathcal{H}(N;c)$ with $c=0$ and $r=\|x_0\|$.
\vspace{1ex}\qed

In the remaining of this section we prove a version of Corollary \ref{cor:cod2snr} for the conformal model $(\R^{n+1}\setminus \R^{k-1}, g)$ of $\Hy^k\times \Sf^{n-k+1}$,  $2\leq k\leq n$. As a particular case,  it provides a classification of the congruence classes of all umbilical  submanifolds  of  $\Hy^n\times \Sf^1$, and hence of $\Hy^n\times \R$. First we make the following observation.

\begin{lemma}\label{eigenbound}
Let $v\in\R^m$ be such that $\|v\|\leq 1$. Then, for any $w\in \R^m$, the smallest eigenvalue of the Gram matrix $G$ of $\{v, w\}$ satisfies $0\leq \bar\lambda\leq 1$, and $\bar \lambda= 1$ if and only if $\<v,w\>=0$, $\|v\|=1$ and $\|w\|\geq 1$.
\end{lemma}
\begin{proof} Given any $w\in \R^m$, let $A\colon \R^2\to \R^m$ be the linear map such that $Ae_1=v$ and $Ae_2=w$, where $\{e_1, e_2\}$ is the canonical basis of $\R^2$. Then the eigenvalues of $G$ are the eigenvalues of $A^tA\colon \R^2\to \R^2$. Thus its smallest eigenvalue $\bar \lambda$ satisfies
$$1\geq \|v\|^2=\<A^tAe_1,e_1\>\geq \mbox{min}\left\lbrace\frac{\<A^tAv,v\>}{\<v,v\>}\;:\:v\neq 0
\right\}=\bar\lambda\geq 0.$$
Moreover, if $\bar\lambda=1$, then $\|v\|=1$ and $e_1$ is an eigenvector of $A^tA$ associated to  $\bar\lambda$. Thus $e_2$ is also an eigenvector of $A^tA$, and hence $\<v, w\>=\<A^tAe_1, e_2\>=0$.
In particular, the eigenvalues of $A^tA$ are $1=\|v\|^2$ and $\|w\|^2$, and hence $\|w\|\geq 1$. The ``if" part is trivial.\qed
\end{proof}

\begin{corollary}\label{cor:cod2case}
Let $p_0\in \R^{n+1}\setminus\R^{k-1}$, $2\leq k\leq n$, be such that $p_0=p_0^\perp$. Let $e\in (\R^{k-1})^\perp$ and $\eta\in\R^{k-1}$  be such that $\|e\|=\|\eta\|=1$ and $\<p_0,e\>=0$. 
Set 
$$x_c=p_0+\sqrt{1+c^2}\,\eta\;\;\;\mbox{and}\;\;\;N_c=\frac{1}{\sqrt{1+c^2}}e+\frac{c}{\sqrt{1+c^2}}\eta, \;c\geq0.$$ 
Then the two-parameter family
\begin{equation}\label{two}\{\Sf^n(x_c,r)\cap \mathcal{H}(N_c,c)\colon c> 0,\;\;0<r<\sqrt{1+c^2}\|p_0\|\}\end{equation}
and the one-parameter families 
\begin{align}\label{one}\{\Sf^n(x_c,\sqrt{1+c^2}\|p_0\|)&\cap \mathcal{H}(N_c,c)\colon c> 0\},\;\;\; \{\Sf^n(p_0,r)\cap\mathcal{H}(e,0)\colon r>0\}\nonumber\\
&\mbox{and}\;\;\;\{\Sf^n(p_0,r)\cap\mathcal{H}(\eta,0)\colon r>0\}
\end{align}
contain, together,  precisely one representative of each congruence class of umbilical non-totally geodesic submanifolds of codimension two in $(\R^{n+1}\setminus\R^{k-1},g)$.
\end{corollary}
\begin{proof}
First notice that in the family \eqref{two} we have $\<x_c^\perp,N_c^\perp\>=0$. Then the Gram matrix of $\{r^{-1}x_c^\perp, N_c^\perp\}$ is always diagonal with distinct eigenvalues $\|p_0\|^2/r^2$ and $(1+c^2)^{-1}$. Moreover, since we are taking $r<\sqrt{1+c^2}\|p_0\|$, then $(1+c^2)^{-1}$ is the smallest one.

Assume that $\Sf^n(x_c,r)\cap \mathcal{H}(N_c,c)$ and $\Sf^n(x_{\tilde{c}},\tilde{r})\cap \mathcal{H}(N_{\tilde{c}},\tilde{c})$ belong to the family \eqref{two} and are congruent. 
Then the eigenvalues of the Gram matrices of $\{r^{-1}x_c^\perp,N_c^\perp\}$ and $\{\tilde{r}^{-1}x_{\tilde{c}}^\perp,N_{\tilde{c}}^\perp\}$, in particular their smallest ones, coincide, and hence $c=\tilde{c}$ and $r=\tilde{r}$. Similarly, one can check that distinct elements  in any of the one-parameter families in \eqref{one} lie in distinct congruence classes, as well as elements of  
distinct families.

Let $\Sf^n(\tilde{x}_0,\tilde{r})\cap \mathcal{H}(\tilde{N},\tilde{c})$ be an umbilical non-totally geodesic submanifold with codimension two of $(\R^{n+1}\setminus\R^{k-1},g)$ and
let $\tilde{G}$ be the Gram matrix of $\{\tilde{r}^{-1}\tilde{x}_0^\perp, \tilde{N}^\perp\}$. Since $\|\tilde{N}^\perp\|\leq 1$,  we see from Lemma \ref{eigenbound} that the smallest eigenvalue $\tilde{\lambda}$ of $\tilde{G}$ satisfies $\tilde{\lambda}\leq 1$. 
If $(\tilde{N}^T,\tilde{c})=(0,0)$, then $\|\tilde{N}^\perp\|= 1$ and $\<\tilde{x}_0^\perp,\tilde{N}\>=0$. Therefore, in this case Proposition \ref{umb1} implies that $\Sf^n(\tilde{x}_0,\tilde{r})\cap \mathcal{H}(\tilde{N},\tilde{c})$ is congruent to $\Sf^n(p_0,r)\cap\mathcal{H}(e,0)$ with $r=\tilde{r}\|p_0\|/\|\tilde{x}_0^\perp\|$.
Notice that, if $\tilde{\lambda}=1$,  then $\<\tilde{x}_0^\perp,\tilde{N}^\perp\>=0$ and $\|\tilde{N}^\perp\|=1$ by the second assertion in Lemma \ref{eigenbound}, implying that $(\tilde{N}^T,\tilde{c})=(0,0)$.

From now on we assume that $(\tilde{N}^T,\tilde{c})\neq(0,0)$, and thus $\tilde{\lambda}<1$.
If $\tilde{\lambda}=0$, then  
$\Sf^n(\tilde{x}_0,\tilde{r})\cap \mathcal{H}(\tilde{N},\tilde{c})$ is congruent to $\Sf^n(p_0,r)\cap\mathcal{H}(\eta,0)$ with $r^2=\|p_0\|^2(\tilde{r}^{-2}\|\tilde{x}_0^\perp\|^2+\|\tilde{N}^\perp\|^2)^{-1}$. 
Now assume that $0<\tilde{\lambda}< 1$. If $\tilde{\lambda}$ has multiplicity $2$, then $\Sf^n(\tilde{x}_0,\tilde{r})\cap \mathcal{H}(\tilde{N},\tilde{c})$ is congruent to $\Sf^n(x_c,\sqrt{1+c^2}\|p_0\|)\cap \mathcal{H}(N_c,c)$ with $c=\sqrt{\tilde{\lambda}^{-1}-1}$.  Finally, if $\tilde{G}$ has two distinct eigenvalues $\tilde{\lambda}<\tilde{\lambda}'$, then $\Sf^n(\tilde{x}_0,\tilde{r})\cap \mathcal{H}(\tilde{N},\tilde{c})$ is congruent to $\Sf^n(x_c,r)\cap \mathcal{H}(N_c,c)$ with
$c=\sqrt{\tilde{\lambda}^{-1}-1}$ and $r^2=\|p_0\|^2/\tilde{\lambda}'$.

As for an umbilical non-totally geodesic  $\mathcal{H}(N_1;c_1)\cap\mathcal{H}(N_2;0)$, first notice that, since it is not totally geodesic, it can not contain $\R^{k-1}$.
Then, its image by an inversion with respect to a hypersphere centered at some point $y\in\R^{k-1}\setminus\mathcal{H}(N_1;c_1)\cap\mathcal{H}(N_2;0)$ is a non-totally geodesic $\Sf^n(\tilde{x}_0,\tilde{r})\cap \mathcal{H}(\tilde{N},\tilde{c})$, which is  congruent to some element of the family \eqref{two} or one of the families in \eqref{one} by the part that has already been proved.
\end{proof}
\begin{remark}\emph{In Corollary \ref{cor:cod2case},  the elements of the families \eqref{two} and  \eqref{one}  that are substantial are those in the family  \eqref{two} and  those of the first family in \eqref{one}. }
\end{remark}

\section{Umbilical submanifolds of $\Hy^k\times \mathbb{S}^{n-k+1}$ as rotational submanifolds.}

The next result shows that any umbilical submanifold  of $\Hy^k\times \mathbb{S}^{n-k+1}$ with substantial codimension $p$ (which is necessarily less than or equal to $\mbox{min}\,\{k+1, n-k+2\}$ by Corollary~\ref{substantial}) is a rotational submanifold in a suitable sense, under mild further restrictions on $p$ and $k$. 

Given a space-like subspace $V\subset  \Le^{n+3}=W_1\oplus W_2\supset \Hy^k\times \mathbb{S}^{n-k+1}$, for $1\leq i\leq 2$ set  $U_i=V^\perp\cap W_i$, and let $U_i^\perp$ be the orthogonal complement of $U_i$ in $W_i$. Denote by $O(U^\perp_1)$ (respectively, $O(U^\perp_2)$) the subgroup of  all elements of $O_1(k+1)$ (respectively, $O(n-k+2)$) that fix $U^\perp_1$ (respectively, $U^\perp_2$) pointwise.

\begin{proposition}\label{prop: rot}  Let $S=(\Hy^k\times \mathbb{S}^{n-k+1})\cap V^\perp$, where $V\subset \Le^{n+3}$ is a space-like subspace of dimension $p$, be a substantial umbilical submanifold with codimension $p$ of $\Hy^k\times \mathbb{S}^{n-k+1}\subset \Le^{n+3}$. Then the following assertions hold:
\begin{itemize}
    \item If $p\leq k-1$ then the subgroup  $\{I\}\times O(U^\perp_2)$ of $O_1(k+1)\times O(n-k+2)$  leaves $S$ invariant and acts on $S$ with cohomogeneity $n-k+1$, the orbit through (a regular point) $(x,y)\in S$ being the $(k-p)$-dimensional umbilical submanifold  $V^\perp \cap   (\mathbb{H}^{k}\times \{y\})$;
    \item If $p\leq n-k$ then the subgroup  $\{I\}\times O(U^\perp_2)$ of $O_1(k+1)\times O(n-k+2)$  leaves $S$ invariant and acts on $S$ with cohomogeneity $k$, the orbit through (a regular point) $(x,y)\in S$ being the $(n-k-p+1)$-dimensional sphere  $V^\perp \cap (\{x\}\times \mathbb{S}^{n-k+1})$.
\end{itemize}
\end{proposition}
\proof  Since $S=(\Hy^k\times \mathbb{S}^{n-k+1})\cap V^\perp$ for some space-like subspace $V\subset \Le^{n+3}$ with dimension $p$, any isometry $T\in O_1(k+1)\times O(n-k+2)$ of  
$\Hy^k\times \mathbb{S}^{n-k+1}$ such that $T(V^\perp)=V^\perp$ leaves $S$ invariant. 
Since $S$ is in substantial codimension $p$, then $V\cap W_i=\{0\}$ for $1\leq i\leq 2$ by Proposition \ref{prop:substantial}, hence $U_i=V^\perp\cap W_i$ has codimension $p$ in $W_i$ for $1\leq i\leq 2$. Therefore, if $p\leq k-1$ (respectively, if $p\leq n-k$), then $\dim U_1\geq 2$ (respectively, $\dim U_2\geq 2$), hence $O(U^\perp_1)\times \{I\}$ (respectively, $\{I\}\times  O(U^\perp_2)$ leaves $S$ invariant, the orbit of  (a regular point)  $(x,y)\in S$ being  the $(k-p)$-dimensional (respectively, $(n-k-p+1)$-dimensional) umbilical submanifold $V^\perp \cap  (\mathbb{H}^k\times \{y\})$ (respectively, sphere $V^\perp \cap (\{x\}\times \mathbb{S}^{n-k+1})$).\qed

\begin{corollary}\label{prop:geometry}
Let $S=(\Hy^k\times \mathbb{S}^{n-k+1})\cap V^\perp$ be as in Proposition \ref{prop: rot}.
\begin{itemize}
\item[(i)] If $p\leq \mbox{min}\,\{k-1, n-k\}$, then
the subgroup $O(U^\perp_1)\times O(U^\perp_2)$ of $O_1(k+1)\times O(n-k+2)$ leaves $S$ invariant  and acts on $S$ with cohomogeneity $p$, 
the orbit through (a regular point) $(x,y)\in S$ being the product 
$(V^\perp \cap  (\mathbb{H}^k\times \{y\})) \times ( V^\perp \cap (\{x\}\times \mathbb{S}^{n-k+1}))$; 

\item[(ii)] If $k=1=p$, then the subgroup $\{I\}\times O(U^\perp_2)$ of $O_1(2)\times O(n+1)$  leaves $S$ invariant and acts on $S$ with cohomogeneity $1$, the orbit through (a regular point) $(x,y)\in S$ being the $(n-1)$-dimensional sphere $V^\perp \cap (\{x\}\times \mathbb{S}^{n})$;
\item[(iii)] If $k=1$, $p=2$ and $n\geq 3$, then the subgroup $\{I\}\times O(U^\perp_2)$  of $O_1(2)\times O(n+1)$  leaves $S$ invariant and acts on $S$ with cohomogeneity $1$, the orbit through (a regular point) $(x,y)\in S$ being the $(n-2)$-dimensional sphere $V^\perp \cap (\{x\}\times \mathbb{S}^{n})$;
\item[(iv)]  If $k=n$ and $p=1$, then the subgroup $O(U^\perp_1)\times \{I\}$ of $O_1(n+1)\times O(2)$  leaves $S$ invariant and acts on $S$ with cohomogeneity $1$, the orbit through (a regular point) $(x,y)\in S$ being the $(n-1)$-dimensional umbilical submanifold $V^\perp \cap (\mathbb{H}^{n}\times \{y\})$;
 \item[(v)] If $k=n\geq 3$ and $p=2$, then the subgroup $O(U^\perp_1)\times \{I\}$ of $O_1(n+1)\times O(2)$  leaves $S$ invariant and acts on $S$ with cohomogeneity $1$, the orbit through (a regular point) $(x,y)\in S$ being the $(n-2)$-dimensional umbilical submanifold $V^\perp \cap (\mathbb{H}^{n}\times \{y\})$.
 \end{itemize}
\
\end{corollary}

 It follows from parts $(iv)$ and $(v)$ of Corollary \ref{prop:geometry}  that  umbilical submanifolds  of $\Hy^n\times \mathbb{S}^{1}$  (or equivalently,   $\Hy^n\times \mathbb{R}$) with substantial codimension one (respectively, substantial codimension two), are rotational submanifolds with curves in $\Hy^1\times \mathbb{S}^{1}$ (respectively, $\Hy^2\times \mathbb{S}^{1}$) as profiles. In the next section these curves are explicitly described.

\subsection{The profile curves of umbilical submanifolds  of $\Hy^n\times \mathbb{S}^{1}$}

We first consider the codimension $1$ case, that is, umbilical non-totally geodesic hypersurfaces of $\Hy^n\times\Sf^1$. Let $(\R^{n+1}\setminus\R^{n-1},g)$ be the conformal model of $\Hy^n\times\Sf^1$ and let $\{e_1,\dots,e_{n+1}\}$ be an orthonormal basis of $\R^{n+1}$ such that $\{e_1,\dots,e_{n-1}\}$ spans $\R^{n-1}$. By Corollary \ref{cor:hypcase}, it is enough to consider the one-parameter family of hyperspheres $\{\Sf^{n}(x_0,r), r>0\}$ centered at a point $x_0$ with $x_0^\perp\neq 0$. For simplicity, we take $x_0=e_{n+1}$.
The hypersphere $\Sf^n(e_{n+1},r)$ is represented by the unit space-like vector $z=\frac{1}{r}\Psi(e_{n+1})+\frac{r}{2}w\in\Le^{n+3}$, whose $W_1$-component is
$$
z^T=\frac{1}{r}v+\frac{1}{2}\left(r-\frac{1}{r}\right)w.
$$
Thus $z^T$ is space-like, light-like or time-like, according to whether  $r>1$, $r=1$ or $r<1$, respectively. Hence, with the notations as in Proposition \ref{prop:geometry}, the subspace $U_1$ is, accordingly, time-like, degenerate or space-like, respectively. We call these cases  \emph{hyperbolic}, \emph{parabolic} or \emph{spherical},
respectively.

Let $\R^2=(\R^{n-1})^\perp=\spa\{e_n,e_{n+1}\}$ and consider the curve $c\colon I\to\R^2$ given by  
$$
c(\theta)=e_{n+1}+r\cos(\theta)e_{n+1}+r\sin(\theta)e_n,
$$
where $I=(-\pi,\pi]$, $I=(-\pi,\pi)$ or $I=[-\pi+\arccos(r),\pi-\arccos(r)]$, depending on whether we are in the hyperbolic, parabolic or spherical case, respectively.
Notice that the image $\Theta(c(I))$ under the diffeomorphism $\Theta\colon \R^{n+1}\setminus \R^{n-1}\to \Hy^n\times \Sf^{1}$ given by \eqref{eq:Theta} is contained in a totally geodesic $\Hy^1\times\Sf^1\subset\Hy^n\times\Sf^1$ and that, for each $\theta\in I$, the point $c(\theta)$ lies in a different slice $\Hy^n\times\{q\}$, $q\in\Sf^1$.
Since the the subgroup of isometries $O(U_1^\perp)\times \{I\}$ acts on each slice, 
and this action is transitive on the intersection $\Theta(\Sf^n(e_{n+1},r))\cap\Hy^n\times\{q\}$ of the umbilical hypersurface $\Theta(\Sf^n(e_{n+1},r))$ of $\Hy^n\times \mathbb{S}^{1}$ with each slice $\Hy^n\times\{q\}$, it follows that $\Theta(\Sf^n(e_{n+1},r))$ is generated by the action of
$O(U_1^\perp)\times \{I\}$ on the curve  $\Theta(c(I))$.  Notice also that if $\Theta(c(\theta))$ lies in the slice $\Hy^n\times\{q\}$, then the orbit $\Theta(\Sf^n(e_{n+1},r))\cap\Hy^n\times\{q\}$ of  $\Theta(c(\theta))$ under $O(U_1^\perp)\times \{I\}$ is an equidistant hypersurface, horosphere or sphere in that slice (which can degenerate into a point), depending on whether we are in the hyperbolic, parabolic or spherical case, respectively.
\vspace{1ex}

Let us now consider the case of umbilical submanifolds with codimension $2$. By Corollary \ref{cor:cod2case}, it is enough to consider the elements of the two families of umbilical submanifolds given by \eqref{two} and \eqref{one} that have substantial codimension $2$. Following the notations in that proposition, set $p_0=e_{n+1}$, $e=e_n$, $\eta=e_1$ and let $x_c$ and $N_c$ be given by
$$
x_c=e_{n+1}+\sqrt{1+c^2}\,e_1\;\;\mbox{and}\;\; N_c=\frac{1}{\sqrt{1+c^2}}e_n+\frac{c}{\sqrt{1+c^2}}e_1,\,\,\, c>0.$$
Let $z_1$ and $z_2$ be the unit space-like vectors associated to  
$\Sf^n(x_c,r)$ and $\mathcal{H}(N_c,c)$, respectively. Then
$\{z_1,z_2\}$ is an orthonormal basis of the space-like subspace
$V\subset \Le^{n+3}$ associated to the umbilical submanifold
$\Sf^n(x_c,r)\cap\mathcal{H}(N_c,c)$. Notice that the $W_1$
components $z_1^T$ and $z_2^T$ are orthogonal, and that $z_2^T$ is
always space-like. Moreover, $\<z_1^T,z_1^T\>=1-\frac{1}{r^2}$.
Thus the subspace $U_1\subset W_1$ is time-like, degenerate or 
space-like depending on whether $r>1$, $r=1$ or $r<1$, respectively. 
We call these cases accordingly \emph{hyperbolic},
\emph{parabolic} or \emph{spherical}.

Consider the curve $c\colon I\to \mathcal{H}(N_c,c)$ given by
$$
c(\theta)=x_c+r\cos(\theta)e_{n+1}+r\sin(\theta)\eta_c,
$$
 where
\be\label{etac}
 \eta_c=-\frac{c}{\sqrt{1+c^2}}e_n+\frac{1}{\sqrt{1+c^2}}e_1.
\ee
Here $I=(-\pi,\pi]$, $I=(-\pi,\pi)$ or $I=[-\pi+\arccos(r),\pi-\arccos(r)]$,
depending on whether we are in the hyperbolic, parabolic or spherical case, respectively.
Notice that $\Theta(c(I))$ is contained in a totally geodesic $\Hy^2\times\Sf^1\subset\Hy^n\times\Sf^1$ and that, for each $\theta\in I$, the point $\Theta(c(\theta))$ lies in a different slice $\Hy^n\times\{q\}$, $q\in\Sf^1$.
Since the the subgroup of isometries $O(U_1^\perp)\times \{I\}$ acts on each slice, and this action is transitive on the intersection $\Theta(\Sf^n(x_c,r)\cap\mathcal{H}(N_c,c))\cap\Hy^n\times\{q\}$ of the umbilical submanifold $\Theta(\Sf^n(x_c,r)\cap\mathcal{H}(N_c,c))$ with each slice $\Hy^n\times\{q\}$, it follows as before that $\Theta(\Sf^n(x_c,r)\cap\mathcal{H}(N_c,c))$ is generated by the action of $O(U_1^\perp)\times \{I\}$ on the curve  
$\Theta(c(I))$. 

As for the umbilical submanifolds in the one-parameter family $\{\Sf^n(x_c,\sqrt{1+c^2})\cap \mathcal{H}(N_c,c)$, $c>0$, they are always hyperbolic. Here we define $c\colon (-\pi,\pi]\to \mathcal{H}(N_c,c)$ by
$$
c(\theta)=x_c+\sqrt{1+c^2}\cos(\theta)e_{n+1}+\sqrt{1+c^2}\sin(\theta)\eta_c,
$$
where $\eta_c$ is given by \eqref{etac}. Then we see as in the previous case that the submanifold $\Theta(\{\Sf^n(x_c,\sqrt{1+c^2})\cap \mathcal{H}(N_c,c))$ is  generated by the action of $O(U_1^\perp)\times \{I\}$ on the curve  $\Theta(c(I))$.

\begin{remark}\emph{A similar description can be given as above of the profile curves of the umbilical submanifolds of $\mathbb{S}^n\times \mathbb{R}$, which is much simpler than those given in \cite{mt1}, \cite{st} and \cite{vv} and, in particular, avoids solving any ODE.}
\end{remark}

\noindent \emph{Acknowledgement.} After a preliminary version of this paper was completed, we learned that a classification of umbilical \emph{hypersurfaces} of $\Hy^n\times \R$ using a different method was provided in \cite{ls}.

\vspace*{5ex}

\noindent Universidade de S\~ao Paulo\\
Instituto de Ci\^encias Matem\'aticas e de Computa\c c\~ao.\\
Av. Trabalhador S\~ao Carlense 400\\
13560-970 -- S\~ao Carlos\\
BRAZIL\\
\texttt{mibieta@impa.br} and \texttt{tojeiro@icmc.usp.br}

\end{document}